%% file: GSfinal.tex
\documentclass[]{interact}

\usepackage{epstopdf}
\usepackage[caption=false]{subfig}

\usepackage[numbers,sort&compress]{natbib}
\bibpunct[, ]{[}{]}{,}{n}{,}{,}
\makeatletter
\def\NAT@def@citea{\def\@citea{\NAT@separator}}
\makeatother
\usepackage{hyperref}
\hypersetup{
	colorlinks=true,
	linkcolor=blue, 
	citecolor=red, 
	urlcolor=blue  } 
\usepackage{mathptmx}   

\theoremstyle{plain}
\newtheorem{theorem}{Theorem}[section]
\newtheorem{lemma}[theorem]{Lemma}
\newtheorem{corollary}[theorem]{Corollary}
\newtheorem{proposition}[theorem]{Proposition}

\theoremstyle{definition}
\newtheorem{definition}[theorem]{Definition}
\newtheorem{example}[theorem]{Example}

\theoremstyle{remark}
\newtheorem{remark}{Remark}

\renewcommand {\theenumi} {\rm\roman{enumi}}

\input{macro}

\newcommand{\NDC}[1]{\todo[inline,color=green!40]{NDC {#1}}}

\begin{document}
	
\title{Generalized Separation of Collections of Sets}

\author{
\name{Nguyen Duy Cuong\textsuperscript{a}, Alexander Y. Kruger\textsuperscript{b}}
\thanks{CONTACT Alexander Y. Kruger. Email: alexanderkruger@tdtu.edu.vn}
\affil{\textsuperscript{a} Department of Mathematics, College of Natural Sciences, Can Tho University, Can Tho, Vietnam;
\textsuperscript{b} Analytical and Algebraic Methods in Optimization Research Group, Faculty of Mathematics and Statistics, Ton Duc Thang University, Ho Chi Minh City, Vietnam}
}
\maketitle

\begin{abstract}
We show that the existing generalized separation statements including the conventional extremal principle and its extensions differ {in the ways  norms on product spaces are defined}.
We prove a general separation statement with arbitrary product norms covering the existing results of this kind.
The proof is divided into a series of claims and exposes the key steps and arguments used when proving generalized separation statements.
As an application, we prove dual necessary (sufficient) conditions for an abstract product norm extension of the approximate stationarity (transversality) property.
\end{abstract}


\begin{amscode}
	49J52; 49J53; 49K40; 90C30; 90C46
\end{amscode}

\setcounter{tocdepth}{2}

\section{Introduction}

In this paper, we discuss generalized separation of collections of sets in a normed vector space in the sense of the dual necessary conditions in the \emph{extremal principle} \cite{KruMor80,MorSha96, Mor06.1}.
In recent years, several abstract generalized separation results characterizing non-intersection of collections of sets and covering the conventional extremal principle and its generalizations have been established.
The main aim of this paper is to clarify the relationship between the existing separation results of this kind.
We prove a more general (and simpler) generalized separation theorem involving abstract norms in product spaces and show that the existing results correspond to choosing appropriate norms in this theorem.

Throughout the paper we consider a collection of $n>1$ arbitrary nonempty subsets $\Omega_1,\ldots,\Omega_n$ of a vector space $X$.
We write $\{\Omega_1,\ldots,\Omega_n\}$ to denote the collection of these sets as a single object,
and use the notation $\widehat\Omega:=\Omega_1\times\ldots\times\Omega_n$ for the ``aggregate'' set in the product space $X^n$.
The space $X$ is usually assumed to be equipped with a norm $\|\cdot\|$, while for norms on product spaces we use the notation $\vertiii{\,\cdot\,}$ (sometimes also $\vertiii{\,\cdot\,}_+$).
We keep the same notations for the corresponding dual norms.
We use $\langle\cdot,\cdot\rangle$ and $d$ to denote, respectively, the bilinear form defining the pairing between the primal and dual spaces and distances (including point-to-set distances) in all spaces determined by the corresponding norms.

The conventional extremal principle gives necessary conditions for the \emph{extremality} of a collection of sets.

\begin{definition}
[Extremality]
\label{D1.1}
The collection $\{\Omega_1,\ldots,\Omega_n\}$
is extremal at $\bx\in\bigcap_{i=1}^n\Omega_i$ if there is a $\rho\in(0,+\infty]$ such that,
for any $\varepsilon>0$, there exist $a_1,\ldots,a_n\in X$ such that
$\max_{1\le i\le n}\|a_i\|<\varepsilon$ and
$\bigcap_{i=1}^n(\Omega_i-a_i)\cap B_\rho(\bx)=\emptyset$.
\end{definition}

Symbol $B_\rho(\bx)$ in the above definition denotes the open ball with centre $\bx$ and radius $\rho$.
For brevity, we combine in this definition the cases of local ($\rho<+\infty$) and global ($\rho=+\infty$) extremality.
In the latter case, the point $\bx$ plays no role apart from ensuring that $\bigcap_{i=1}^n\Omega_i\ne\es$.
It is easy to see that the property does not change if $B_\rho(\bx)$ is replaced by the corresponding
closed ball $\overline B_\rho(\bx)$.

The extremal principle assumes that $(X,\|\cdot\|)$ is Asplund and the sets are closed, and gives necessary conditions in terms of elements of the (topologically) dual space $X^*$.

\begin{theorem}
[Extremal principle]
\label{EP}
Let $\Omega_1,\ldots,\Omega_n$ 
be closed subsets of an Asplund space ${(X,\|\cdot\|)}$.
If $\{\Omega_1,\ldots,\Omega_n\}$
is extremal at $\bx\in\bigcap_{i=1}^n\Omega_i$, then,
for any $\varepsilon>0$, there exist $x_i\in\Omega_i\cap B_\eps(\bx)$ and $x_i^*\in X^*$ $(i=1,\ldots,n)$ such that
$d(x_i^*,N^F_{\Omega_i}(x_i))<\eps$
$(i=1,\ldots,n)$, and
\begin{gather}
\label{EP-3}
\sum_{i=1}^n x_i^*=0,\quad
\sum_{i=1}^{n}\|x_i^*\|=1.
\end{gather}
\end{theorem}

Theorem~\ref{EP} employs \emph{\Fr\ normal cones};
see the definition in Section~\ref{Prel}.
The dual conditions are formulated in a \emph{fuzzy} form and can be interpreted as \emph{generalized separation}; cf. \cite{BuiKru18,BuiKru19}.
The statement naturally yields the limiting version of the extremal principle in terms of \emph{limiting normal cones} \cite{KruMor80,Kru85.1,MorSha96,Kru03, Mor06.1}.
In finite dimensions, the limiting version comes for free, while in infinite dimensions, additional \emph{sequential normal compactness} type assumptions are required to ensure that the second equality in \eqref{EP-3} is preserved when passing to limits; cf. \cite{Mor06.1}.

The extremality property in Definition~\ref{D1.1} embraces many conventional and generalized optimality notions.
It assumes neither convexity of the sets nor that any of the sets has nonempty interior.
The generalized separation conditions in Theorem~\ref{EP} naturally translate into necessary
optimality conditions (multiplier rules) and various subdifferential and coderivative calculus results in nonconvex settings \cite{KruMor80,Kru85.1,Kru03,Mor06.1,Mor06.2,BorZhu05}.

The conventional extremal principle was established in this form in \cite{MorSha96} as an extension of the original result from \cite{KruMor80}, which had been formulated in \emph{\Fr\ smooth} spaces and,
{by analogy with \cite{DubMil65},}
referred to as the \emph{generalized Euler equation}.
The proof employs two fundamental results of variational analysis: the \emph{\EVP} \cite{Eke74} and fuzzy \emph{\Fr\ subdifferential sum rule} due to Fabian \cite{Fab89}; see Lemmas~\ref{EVP} and \ref{SR}\,\eqref{SR.Fr}.
Recall that a Banach space is {Asplund} if every continuous convex function on an open convex set is Fr\'echet differentiable on a dense subset, or equivalently, if the dual of each its separable subspace is separable.
We refer the reader to \cite{Phe93,Mor06.1,BorZhu05} for discussions about and characterizations of Asplund spaces.
All reflexive, particularly, all finite dimensional Banach spaces are Asplund.
It was also shown in \cite{MorSha96} that the necessary conditions in Theorem~\ref{EP} are actually equivalent to the Asplund property of the space (see also \cite[Theorem~2.20]{Mor06.1}).

The statement can be easily extended to general Banach spaces.
This only requires replacing the application of the \Fr\ subdifferential sum rule in the conventional proof by the Clarke subdifferential sum rule \cite{Cla83}.
Other normal cones corresponding to subdifferentials possessing a satisfactory calculus in Banach spaces can be employed instead of the Clarke ones, e.g., Ioffe normal cones \cite{Iof17}; cf. \cite[Remark~2.1\,(iii)]{BuiKru19}.

The key point of Definition~\ref{D1.1} of extremality is the replacement of the given  collection $\{\Omega_1,\ldots,\Omega_n\}$ with the property $\bigcap_{i=1}^n\Omega_i\ne\es$ by another collection $\{\Omega_1',\ldots,\Omega_{n+1}'\}$ satisfying $\bigcap_{i=1}^{n+1}\Omega_i'=\es$.
The first $n$ sets of the new collection are linear translations of the original ones: $\Omega_i':=\Omega_i-a_i$ $(i=1,\ldots,n)$, and the last one is a \nbh\ of the given point $\bx$: $\Omega_{n+1}':=B_{\rho}(\bx)$.
(To ensure that all new sets are closed when the original sets are, it may be convenient to use a closed \nbh\
$\overline B_{\rho}(\bx)$.)

The proof of the extremal principle is applicable with minor changes to more general deformations of the original sets than linear translations.
This observation allowed Mordukhovich et al. \cite{MorTreZhu03} (see also \cite{BorZhu05,Mor06.2}) to present a more flexible extended version of the extremal principle with the definition of extremality using deformations of the sets by set-valued mappings.
This extension has proved useful when studying various multiobjective problems \cite{MorTreZhu03, ZheNg06,LiNgZhe07,Bao14.2}.

As noted above, the core of the proof of the extremal principle in \cite{KruMor80,MorSha96,MorTreZhu03} (as well as in numerous subsequent versions) consists of deriving dual necessary conditions for non-intersection of a certain collection of closed sets.
It was observed by Zheng \& Ng in \cite{ZheNg05.2} that such conditions can be of interest on their own as powerful tools for proving various `extremal' results (see, e.g., the \emph{sequential extremal principle} in \cite{CuoKru}).
They formulated the core of the proof of the extremal principle as two abstract non-inter\-section lemmas
and demonstrated their other applications in optimization.
This fruitful idea has been further refined in the \emph{unified separation theorem} in \cite{ZheNg11}.
The next theorem basically combines the statements of \cite[Theorems~3.1 and 3.4]{ZheNg11}.
Note that, unlike Theorem~\ref{EP}, the number $\eps$ in the next theorem is fixed.

\begin{theorem}
\label{ZhNg}
Let $\Omega_1,\ldots,\Omega_n$ 
be closed subsets of a Banach space ${(X,\|\cdot\|)}$,
$\omega_i\in\Omega_i$ $(i=1,\ldots,n)$,
$\eps>0$ and $\de>0$.
Suppose that
$\bigcap_{i=1}^n\Omega_i=\emptyset$ and
\begin{gather}
\label{T3.1-1}
\max_{1\le i\le n-1} \|\omega_i-\omega_n\| <\inf_{u_i\in\Omega_i\;(i=1,\ldots,n)}\max_{1\le i\le n-1}\|u_i-u_n\| +\varepsilon.
\end{gather}
The following assertions hold true:
\begin{enumerate}
\item
there exist $x_i\in\Omega_i\cap B_\de(\omega_i)$ and $x_i^*\in X^*$ $(i=1,\ldots,n)$ such that
\begin{gather}
\label{ZhNg-2}
\sum_{i=1}^n x_i^*=0,\quad
\sum_{i=1}^{n-1}\|x_i^*\|=1,
\\
\label{ZhNg-3}
\sum_{i=1}^nd\left(x_i^*,N^C_{\Omega_i}(x_i)\right) <\frac{\eps}{\de},
\\
\label{ZhNg-4}
\sum_{i=1}^{n-1}\langle x^*_i,x_n-x_i\rangle =\max_{1\le i \le n-1} \norm{x_n-x_i};
\end{gather}
\item
if ${(X,\|\cdot\|)}$ is Asplund,
then, for any $\tau\in(0,1)$, there exist $x_i\in\Omega_i\cap B_\de(\omega_i)$ and $x^*_i\in X^*$ $(i=1,\ldots,n)$ such that conditions \eqref{ZhNg-2} are satisfied, and
\begin{gather}
\label{ZhNg-5}
\sum_{i=1}^nd\left(x_i^*,N^F_{\Omega_i}(x_i)\right) <\frac{\eps}{\de},
\\
\label{ZhNg-6}
\sum_{i=1}^{n-1}\langle x^*_i,x_n-x_i\rangle >\tau\max_{1\le i \le n-1}\norm{x_n-x_i}.
\end{gather}
\end{enumerate}
\end{theorem}
\if{
\AK{13/09/24.
To discuss the ``non-intersect index'' and the ``additional conditions'' in Theorem~\ref{ZhNg}.
Anything else from \cite{ZheNg11} to be mentioned?
Have there been more advances after \cite{ZheNg11}?}
\NDC{19/9/24.
A part from the `non-intersect index' and the `additional conditions', they applied the results to `approximate projections' and set-valued optimization.
The paper \cite{ZheYanZou17} studies an exact form of Theorem~\ref{ZhNg} under a compactness assumption.}
}\fi

\begin{remark}
\label{R1}
\begin{enumerate}
\item
\label{R1.1}
The expression $\inf\limits_{u_i\in\Omega_i\;(i=1,\ldots,n)} \max\limits_{1\le i\le n-1}\|u_i-u_n\|$ in the \RHS\ of \eqref{T3.1-1} corresponds to taking $p=+\infty$ in the definition of the \emph{$p$-weighted nonintersect index} \cite[p.~890]{ZheNg11} of $\Omega_1,\ldots,\Omega_n$:
\sloppy
\begin{gather*}
\ga_p(\Omega_1,\ldots,\Omega_n):=
\inf_{u_i\in\Omega_i\;(i=1,\ldots,n)} \Big(\sum_{i=1}^{n-1}\|u_i-u_n\|^p\Big)^{\frac1p}.
\end{gather*}
It is easy to see that $\ga_p(\Omega_1,\ldots,\Omega_n)=0$ when $\bigcap_{i=1}^n\Omega_i\ne\emptyset$.
Employing the general $p$-wei\-gh\-ted index leads to obvious changes in the statement: replacing the maximum of the norms and the sums of the distances by the, respectively, $p$-weighted and $q$-weighted sums with $\frac1p+\frac1q=1$.
\item
Conditions \eqref{ZhNg-4} and \eqref{ZhNg-6} relate the dual vectors $x_i^*$ and the primal vectors $x_n-x_{i}$ ${(i=1,\ldots,n-1)}$.
Such conditions, though not common in the conventional formulations of the extremal principle/generalized separation statements, provide some additional characterizations of the properties.
Conditions of this kind first appeared explicitly in the unified separation theorem in \cite{ZheNg11}, where the authors also provided motivations for employing such conditions.
\end{enumerate}
\end{remark}

Theorem~\ref{EP} and its extension in \cite{MorTreZhu03} are immediate consequences of Theorem~\ref{ZhNg}, see \cite[Remark on p.~1161]{ZheNg05.2} and \cite[Remark on p.~897]{ZheNg11}.

The last set in the collection $\{\Omega_1,\ldots,\Omega_n\}$ plays a special role in Theorem~\ref{ZhNg}; see assumption \eqref{T3.1-1}.
(Of course, this can be any of the sets.)
The last set can indeed be special as it can be seen from Definition~\ref{D1.1}, where the last set is a \nbh\ of the given point.
However, this fact is not reflected in conditions $x_i\in\Omega_i\cap B_\de(\omega_i)$ $(i=1,\ldots,n)$, \eqref{ZhNg-3} and \eqref{ZhNg-5}.
A slightly more advanced version of generalized separation has been given (without proof) in \cite[Lemma~2.1]{CuoKruTha24} and employed in \cite{CuoKruTha24,CuoKruTha25} when characterizing extremality of collections of abstract families of sets.


\begin{theorem}
\label{T1.4}
Let $\Omega_1,\ldots,\Omega_n$ 
be closed subsets of a Banach space ${(X,\|\cdot\|)}$,
$\omega_i\in\Omega_i$ ${(i=1,\ldots,n)}$, $\eps>0$, $\de>0$ and $\eta>0$.
Suppose that
$\bigcap_{i=1}^n\Omega_i=\emptyset$ and
condition \eqref{T3.1-1} is satisfied.
The following assertions hold true:
\begin{enumerate}
\item
there exist $x_i\in\Omega_i$ and $x_i^*\in X^*$ $(i=1,\ldots,n)$ such that conditions \eqref{ZhNg-2} and \eqref{ZhNg-4} are satisfied, and
\begin{gather}
\label{T1.4-1}
\|x_i-\omega_i\|<\de\;\;
(i=1,\ldots,n-1),\quad
\|x_n-\omega_n\|<\eta,
\\
\label{T5.1-2}
\de\sum_{i=1}^{n-1} d\left(x_i^*,N^C_{\Omega_i}(x_i)\right)+ \eta d\left(x_n^*,N^C_{\Omega_n}(x_n)\right) <\eps;
\end{gather}
\item
if $(X,\|\cdot\|)$ is Asplund, then, for any $\tau\in(0,1)$, there exist $x_i\in\Omega_i$ and $x_i^*\in X^*$ ${(i=1,\ldots,n)}$ such that conditions \eqref{ZhNg-2}, \eqref{ZhNg-6} and \eqref{T1.4-1} are satisfied, and
\begin{gather}\label{T5.1-3}
\de\sum_{i=1}^{n-1} d\left(x_i^*,N^F_{\Omega_i}(x_i)\right)+ \eta d\left(x_n^*,N^F_{\Omega_n}(x_n)\right) <\eps.
\end{gather}
\end{enumerate}
\end{theorem}

Theorem~\ref{ZhNg} is obviously a particular case of Theorem~\ref{T1.4} with $\de=\eta$.
It is not difficult to establish an $n$-parameter extension of Theorem~\ref{T1.4}.
Below we show that Theorems~\ref{ZhNg} (including its version involving the $p$-weighted nonintersect index mentioned in Remark~\ref{R1}\,\eqref{R1.1}) and \ref{T1.4} as well as their multi-parameter extensions actually work with sets in product spaces and differ in the ways norms on product spaces are defined.
We prove a general statement with arbitrary product norms covering all the mentioned results.
The statement follows the structure of Theorems~\ref{ZhNg} and \ref{T1.4}, clarifying all the conditions involved in them.
The proof follows that of the core part of the conventional extremal principle and exposes the key steps and arguments used when proving generalized separation statements of this kind.

The next Section~\ref{Prel} collects some general definitions and facts used throughout the paper.
In particular, it contains a discussion of compatibility conditions for norms on product spaces.
Section~\ref{S3} is devoted to generalized separation statements with arbitrary product norms covering all the discussed results.
As an application, we prove in Section~\ref{S4} dual necessary (sufficient) conditions for an abstract product norm extension of the approximate stationarity (transversality) property from \cite{Kru98,Kru05,Kru09,KruTha13,BuiKru19,CuoKru21.2}.

\section{Preliminaries}
\label{Prel}

\paragraph*{Normal cones and subdifferentials.}
We first recall the definitions of normal cones and subdifferentials in the sense of Fr\'echet and Clarke;
see, e.g., \cite{Cla83,Kru03,Mor06.1}.
Given a subset $\Omega$ of a normed space $X$ and a point $\bx\in \Omega$, the sets
\begin{gather}\label{NC}
N_{\Omega}^F(\bx):= \Big\{x^\ast\in X^\ast\mid
\limsup_{\Omega\ni x{\rightarrow}\bar x,\;x\ne\bx} \frac {\langle x^\ast,x-\bx\rangle}
{\|x-\bx\|} \le 0 \Big\},
\\\label{NCC}
N_{\Omega}^C(\bx):= \left\{x^\ast\in X^\ast\mid
\ang{x^\ast,z}\le0
\;\;
\text{for all}
\;\;
z\in T_{\Omega}^C(\bx)\right\}
\end{gather}
are, respectively, the \emph{Fr\'echet} and \emph{Clarke normal cones} to $\Omega$ at $\bx$.
Symbol $T_{\Omega}^C(\bx)$
in \eqref{NCC}
stands for the \emph{Clarke tangent cone} to $\Omega$ at $\bx$:
\begin{multline*}
T_{\Omega}^C(\bx):= \big\{z\in X\mid
\forall x_k{\rightarrow}\bx,\;x_k\in\Omega,\;\forall t_k\downarrow0,\;\exists z_k\to z
\;\;
\text{with}
\;\;
x_k+t_kz_k\in \Omega
\;\;
\text{for all}
\;\;
k\in\N\big\}.
\end{multline*}
The sets \eqref{NC} and \eqref{NCC} are nonempty
closed convex cones satisfying $N_{\Omega}^F(\bx)\subset N_{\Omega}^C(\bx)$.
If $\Omega$ is a convex set, they reduce to the normal cone in the sense of convex analysis:
\begin{gather*}\label{CNC}
N_{\Omega}(\bx):= \left\{x^*\in X^*\mid \langle x^*,x-\bx \rangle \leq 0
\;\;
\text{for all}
\;\;	
x\in \Omega\right\}.
\end{gather*}

For an extended-real-valued function $f:X\to\R_\infty:=\R\cup\{+\infty\}$ on a normed space $X$,
its domain and epigraph are defined,
respectively, by
$\dom f:=\{x \in X\mid f(x) < +\infty\}$
and
$\epi f:=\{(x,\alpha) \in X \times \mathbb{R}\mid f(x) \le \alpha\}$.
The \emph{Fr\'echet} and \emph{Clarke subdifferentials} of $f$ at $\bar x\in\dom f$
are defined, respectively, by
\begin{gather}\label{SF}
	\sd^F f(\bar x):= \left\{x^* \in X^*\mid (x^*,-1) \in N^F_{\epi f}(\bar x,f(\bar x))\right\},
	\\ \label{SC}
	\partial^C{f}(\bx):= \left\{x^\ast\in X^\ast\mid
	(x^*,-1)\in N_{\epi f}^C(\bx,f(\bx))\right\}.
\end{gather}
The sets \eqref{SF} and \eqref{SC} are closed and convex, and satisfy
$\partial^F{f}(\bx)\subset\partial^C{f}(\bx)$.
If $f$ is convex, they
reduce to the subdifferential $\partial{f}(\bx)$ in the sense of convex analysis.
We set $N_{\Omega}^F(\bx)=N_{\Omega}^C(\bx):=\es$ if $\bx\notin \Omega$ and $\partial^F{f}(\bx)=\partial^C{f}(\bx):=\es$ if $\bx\notin\dom f$.
It holds: $N_{\Omega}^F(\bx)=\partial^Fi_\Omega(\bx)$ and $N_{\Omega}^C(\bx)=\partial^Ci_\Omega(\bx)$, where $i_\Omega$ is the \emph{indicator function} of $\Omega$: $i_\Omega(x)=0$ if $x\in \Omega$ and $i_\Omega(x)=+\infty$ if $x\notin \Omega$.

\paragraph*{Ekeland variational principle and subdifferential sum rules.}
The main tools used in the proofs of generalized separation results are
the celebrated Ekeland variational principle and conventional subdifferential sum rules; see, e.g.,
\cite{Eke74,Cla83,Fab89,Kru03,Mor06.1}.

\begin{lemma}\label{EVP}
Let $(X,d)$ be a complete metric space, $f: X\to\R_{\infty}$ be lower semicontinuous, $x_0\in X$, $\varepsilon>0$ and $\lambda>0$.
If $f(x_0)<\inf_{X} f+\varepsilon$, then there exists an $x\in X$ such that
$d(x,x_0)<\lambda$,
$f(x)\le f(x_0)$, and
$f(u)+(\varepsilon/\lambda)d(u,x){\ge f(x)}$ for all $u\in X.$
\end{lemma}

\begin{lemma}\label{SR}
Let $(X,\|\cdot\|)$ be a Banach space, $f_1,f_2\colon X\to\R_\infty$,
and $\bar x\in\dom f_1\cap\dom f_2$.
Suppose that $f_1$ is Lipschitz continuous near $\bar x$
and $f_2$ is lower semicontinuous near $\bar x$.
Then
\begin{enumerate}
\item
\label{SR.Cl}
$\sd^C(f_1+f_2)(\bar x)\subset \sd^Cf_1(\bar x)+\sd^C f_2(\bar x)$;
\item
\label{SR.convex}
if $f_1$ and $f_2$ are convex, then
$\sd(f_1+f_2)(\bar x)=\sd f_1(\bar x)+\sd f_2(\bar x)$;
\item
\label{SR.Fr}
if $(X,\|\cdot\|)$ is Asplund, then, for any $x^*\in{\sd^F}(f_1+f_2)(\bar x)$ and $\varepsilon>0$, there exist $x_1,x_2\in X$
such that $\|x_i-\bx\|<\varepsilon$, $|f_i(x_i)-f_i(\bx)|<\varepsilon$ $(i=1,2)$, and
$d\big(x^*,\partial^F f_1(x_1)+\partial^F f_2(x_2)\big){<\varepsilon}.$
\end{enumerate}
\end{lemma}

\paragraph*{Product norms.}
As discussed in the Introduction, the proofs of the conventional extremal principle and its extensions involve working with norms and distances in product spaces.
They are also implicitly involved in the definitions and statements.
For instance, Definition~\ref{D1.1} and Theorem~\ref{EP} (implicitly) employ the ``maximum'' norm on $X^n$ given by
\begin{gather}
\label{norm}
\vertiii{(u_1,\ldots,u_{n})}:=\max_{1\le i\le n} \|u_i\|
\;\;
\text{for all}
\;\;
u_1,\ldots,u_{n}\in X.
\end{gather}	
Theorems~\ref{ZhNg} and \ref{T1.4} use norms on $X^{n-1}$ and $X^n$, which can be different from the conventional maximum norm.

In this paper, we consider the spaces $X$, $X^{n-1}$ and $X^n$ (sometimes also $X^{n+1}$) equipped with their own norms.
We use symbols $\vertiii{\,\cdot\,}$ and $\vertiii{\,\cdot\,}_+$ to denote norms on product spaces.
The first one is normally used in $X^{n-1}$ (or $X^n$) while the latter in $X^{n}$ (or $X^{n+1}$).
Given a $u\in X$, we write $(u,\ldots,u)_n$ to specify that $(u,\ldots,u)\in X^n$.
In some statements, we assume certain compatibility between these norms as well as between the corresponding dual norms.
The following compatibility conditions are in use:
{\renewcommand\theenumi{C\arabic{enumi}}
\begin{enumerate}
\item\label{C1}
${n>1}$ and $\max\{\vertiii{(u_1,\ldots,u_{n-1})}, \vertiii{(u_n,\ldots,u_{n})_{n-1}}\}\le \kappa\vertiii{(u_1,\ldots,u_n)}_+$
for some ${\kappa>0}$ and all $u_1,\ldots,u_n\in X$;
\item\label{C2}
${n>1}$ and $\max\{\vertiii{(u_1,\ldots,u_{n-1})}, \|u_{n}\|\}\le \kappa\vertiii{(u_1,\ldots,u_n)}_+$
for some $\kappa>0$ and all $u_1,\ldots,u_n\in X$;
\item\label{C3}
$\max_{1\le i\le n}\|u_i\|\le\kappa\vertiii{(u_1,\ldots,u_n)}$
for some $\kappa>0$ and all $u_1,\ldots,u_n\in X$;
\item\label{C4}
$\vertiii{(u_1,\ldots,u_n)}\le\kappa\max_{1\le i\le n}\|u_i\|$
for some $\kappa>0$ and all $u_1,\ldots,u_n\in X$;
\item\label{C5}
$\vertiii{(u,\ldots,u)_n}\le\kappa\|u\|$
for some $\kappa>0$ and all $u\in X$;
\item\label{C6}
$\sum_{i=1}^n\|u_i^*\|\le \kappa\vertiii{(u_1^*,\ldots,u_n^*)}$ for some $\kappa>0$ and all $u_1^*,\ldots,u_n^*\in X^*${;}
{\item\label{C7}
$\vertiii{(u^*_1,\ldots,u^*_n)}\le\kappa \sum_{i=1}^{n}\|u^*_i\|$ for some $\kappa>0$ and all $u^*_1,\ldots,u^*_n\in X^*$.
}
\end{enumerate}
}
Note that the last two conditions refer to dual norms.
It is easy to check that all the above conditions are satisfied for the maximum norm \eqref{norm} (and the corresponding dual sum norm), and more generally, for $p$-norms ($p\in[1,+\infty]$):
\begin{gather}
\label{pnorm}
\vertiii{(u_1,\ldots,u_n)}:= \Big(\sum_{i=1}^{n}\|u_i\|^p\Big)^{\frac{1}{p}}
\;\;
\text{for all}
\;\;
u_1,\ldots,u_{n}\in X,
\end{gather}
and their combinations.
Moreover,
{all the conditions are satisfied in finite dimensions since both sides of the inequalities in \eqref{C1}--\eqref{C7} are actually norms in appropriate finite-dimensional spaces\footnote{Observed by a reviewer}.}

In the case of \eqref{pnorm}, if $(X,\|\cdot\|)$ is Banach (Asplund), then $(X^n,\vertiii{\,\cdot\,})$ is also Banach (Asplund).
The $p$-norm \eqref{pnorm} is a composition of the mapping
$(u_1,\ldots,u_n)\mapsto(\|u_1\|,\ldots,\|u_n\|)$
and the $p$-norm on $\R^n$, which explains its good compatibility properties.
Note that this is not the case for an arbitrary norm on $\R^n$: without assuming certain monotonicity of the second mapping, the composition may not be a norm; see, e.g., \cite[p.134]{LooSte90}.
Below we provide a short proof of the fact for completeness.

\begin{proposition}
\label{P2.4}
Let $(X,\|\cdot\|)$ be a normed space, and $\|\cdot\|_n$ be an arbitrary norm on $\R^n$ such that
$\|(\al_1,\ldots,\al_n)\|_n\le \|(\be_1,\ldots,\be_n)\|_n$ if $\abs{\al_i}\le\abs{\be_i}$ for all $i=1,\ldots,n$.
Then the mapping
\begin{gather*}
(u_1,\ldots,u_n)\mapsto
\vertiii{\,(u_1,\ldots,u_n)\,}:= \|(\|u_1\|,\ldots,\|u_n\|)\|_n
\end{gather*}
is a norm on $X^n$.
\end{proposition}	
\if{
\NDC{23/12/24.
The proposition is not new as observed (without proof) in the calculus book I sent you.
However, I think it is not bad to have the proof explicitly.
}
}\fi
\begin{proof}
We only need to show the triangle inequality since the other conditions are trivially satisfied.
For any $(x_1,\ldots,x_n),(u_1,\ldots,u_n)\in X^n$, we have
\begin{align*}
\vertiii{\,(x_1,\ldots,x_n)+(u_1,\ldots,u_n)\,}
&=\|(\|x_1+u_1\|,\ldots,\|x_n+u_n\|)\|_n\\
&\le
\|(\|x_1\|+\|u_1\|,\ldots,\|x_n\|+\|u_n\|)\|_n\\
&=\|((\|x_1\|,\ldots,\|x_n\|)+(\|u_1\|,\ldots,\|u_n\|))\|_n\\
&\le \|(\|x_1\|,\ldots,\|x_n\|)\|_n+ \|(\|u_1\|,\ldots,\|u_n\|)\|_n\\
&=\vertiii{\,(x_1,\ldots,x_n)\,}+ \vertiii{\,(u_1,\ldots,u_n)\,},
\end{align*}
which completes the proof.
\end{proof}	

The monotonicity of the norm $\|\cdot\|_n$ in Proposition~\ref{P2.4} is essential to ensure that the composition mapping is a norm.

\begin{example}
Consider $(\R,|\cdot|)$ and $(\R^2,\|\cdot\|_2)$, where $\|(\al,\be)\|_2:=|\al-\be|+|\beta|$ for all $(\al,\be)\in\R^2$.
First, observe that $\|\cdot\|_2$ is a norm.
It does not satisfy the monotonicity property in Proposition~\ref{P2.4}.
Indeed, let $(\al_1,\be_1):=(1,-1)$ and $(\al_2,\be_2):=(2,1)$.
Then $|\al_1|<|\al_2|$, $|\be_1|=|\be_2|$, while
$\|(\al_1,\be_1)\|_2=3>2= \|(\al_2,\be_2)\|_2$.
For all $(\al,\be)\in\R^2$, set $\vertiii{\,(\al,\be)\,}:=\|(|\al|,|\be|)\|_2$, and observe that $\vertiii{\,\cdot\,}$ is not a norm on $\R^2$ as it violates the triangle inequality.
Indeed, let $(\al_1,\be_1):=(1,1)$ and $(\al_2,\be_2):=(-1,1)$.
Then $\vertiii{\,(\al_1,\be_1)\,}=\vertiii{\,(\al_2,\be_2)\,}=1$, while $\vertiii{\,(\al_1,\be_1)+(\al_2,\be_2)\,}=\|(0,2)\|_2=4$.
\end{example}
\if{
\AK{16/11/24.
Are the above proposition and example new?}
\NDC{23/23/24.
The previous example was not new.
I took it from the website (that I sent you some time ago) and planned to change it later.
I have made some modifications to make it look new.
}
\AK{26/11/24.
Proper references must be provided.}
}\fi
The next proposition summarizes relations between the compatibility properties.

\begin{proposition}
\label{P02.3}
The following relations hold true:
\begin{enumerate}
\item\label{P2.5-1}
\eqref{C2} $\&$ \eqref{C5} $\Rightarrow$ \eqref{C1};
\if{
If norms $\vertiii{\,\cdot\,}$ and $\vertiii{\,\cdot\,}'$ satisfy \eqref{C2}, and norms $\vertiii{\,\cdot\,}$ and $\|\cdot\|$ satisfy \eqref{C5}, then norms $\vertiii{\,\cdot\,}$ and $\vertiii{\,\cdot\,}'$ satisfy \eqref{C1}.
}\fi
\item\label{P2.5-2}
\eqref{C1} $\&$ \eqref{C3} $\Rightarrow$ \eqref{C2};
\if{
If norms $\vertiii{\,\cdot\,}$ and $\vertiii{\,\cdot\,}'$ satisfy \eqref{C1}, and norms $\vertiii{\,\cdot\,}$ and $\|\cdot\|$ satisfy \eqref{C3}, then norms $\vertiii{\,\cdot\,}$ and $\vertiii{\,\cdot\,}'$ satisfy \eqref{C2}.
}\fi
\item
\label{P02.3.3}
\eqref{C4} $\Rightarrow$ \eqref{C5} $\&$ \eqref{C6}{;}
{\item\label{P2.5-4}
\eqref{C3} $\Rightarrow$ \eqref{C7}.
}
\if{
If norms $\vertiii{\,\cdot\,}$ and $\|\cdot\|$ satisfy \eqref{C4}, then they satisfy \eqref{C5}, and the corresponding dual norms satisfy \eqref{C6}.
}\fi
\end{enumerate}
\end{proposition}

\begin{proof}
Assertions \eqref{P2.5-1} and \eqref{P2.5-2}, and the implication \eqref{C4} $\Rightarrow$ \eqref{C5} in \eqref{P02.3.3} are straightforward.
We show that  \eqref{C4} $\Rightarrow$ \eqref{C6}.
Let $\|\cdot\|$ and $\vertiii{\,\cdot\,}$  satisfy \eqref{C4} with some $\kappa>0$, and $u_1^*,\ldots,u_n^*\in X^*$.
Then
\begin{align*}
\vertiii{(u_1^*,\ldots,u_n^*)}= \sup_{\vertiii{(u_1,\ldots,u_n)}=1} \sum_{i=1}^n\ang{u_i^*,u_i}\ge&\kappa\iv \sup_{\max\limits_{1\le i\le n}\|u_i\|=1} \sum_{i=1}^n\ang{u_i^*,u_i}
\\=&
\kappa\iv \sum_{i=1}^n\sup_{\|u_i\|=1} \ang{u_i^*,u_i}=
\kappa\iv \sum_{i=1}^n\|u_i^*\|.
\end{align*}
{The proof of assertion \eqref{P2.5-4} is similar.}
\end{proof}

\begin{remark}
\if{
\item
Similar to the above proof,
it is not difficult to show that condition \eqref{C3} implies the following counterpart of condition \eqref{C6}: $\vertiii{(u^*_1,\ldots,u^*_n)}\le\kappa \sum_{i=1}^{n}\|u^*_i\|$ for some $\kappa>0$ and all $u^*_1,\ldots,u^*_n\in X^*$.
}\fi
As one can see from the results in Section~\ref{S3}, the norm $\vertiii{\,\cdot\,}$ on $X^{n-1}$ plays a key role in generalized separation statements; the role of the norm $\vertiii{\,\cdot\,}_+$ on $X^{n}$ is more modest:
in many cases, it would be sufficient to consider $\vertiii{\,\cdot\,}_+$ defined for some $\ga>0$ as follows:
\begin{gather}
\label{ganorm}
\vertiii{(u_1,\ldots,u_n)}_+:= \max\{\vertiii{(u_1,\ldots,u_{n-1})},\ga\|u_n\|\}
\;\;
\text{for all}
\;\;
u_1,\ldots,u_{n}\in X.
\end{gather}
Then condition \eqref{C2} is automatically satisfied, and the properties of the norm $\vertiii{\,\cdot\,}_+$ are determined by those of $\|\cdot\|$ and $\vertiii{\,\cdot\,}$.
In particular, $(X^n,\vertiii{\,\cdot\,}_+)$ is Banach (Asplund) whenever $(X,\|\cdot\|)$ and $(X^{n-1},\vertiii{\,\cdot\,})$ are.
\end{remark}

Compatibility conditions \eqref{C3} and \eqref{C4} ensure, in particular, that conditions formulated in terms of tangent or normal cones to the aggregate set $\widehat\Omega:=\Omega_1\times\ldots\times\Omega_n$ can be reformulated in terms of the corresponding cones to the individual sets.

\begin{proposition}
\label{P2.3}
Let $(X,\|\cdot\|)$ and $(X^n,\vertiii{\,\cdot\,})$ be normed spaces, $\|\cdot\|$ and $\vertiii{\,\cdot\,}$ be compatible in the sense of \eqref{C3} and \eqref{C4}, and $\hat x:=(x_1,\ldots,x_n)\in\widehat\Omega$.
The following assertions hold true:
\begin{enumerate}
\item\label{P2.6-1}
$T^C_{\widehat\Omega}(\hat x)= T^C_{\Omega_1}(x_1)\times\ldots\times T^C_{\Omega_n}(x_n)$;
\item
$N^C_{\widehat\Omega}(\hat x)= N^C_{\Omega_1}(x_1)\times\ldots\times N^C_{\Omega_n}(x_n)$;
\item
$N^F_{\widehat\Omega}(\hat x)= N^F_{\Omega_1}(x_1)\times\ldots\times N^F_{\Omega_n}(x_n)$.
\end{enumerate}
\end{proposition}

\begin{proof}
\begin{enumerate}
\item
Let $\hat z:=(z_1,\ldots,z_n)\in T^C_{\widehat\Omega}(\hat x)$ and $i\in\{1,\ldots,n\}$.
Take arbitrary sequences $t_k\downarrow 0$ and $x^k\to x_i$ with $x^k\in\Omega_i$ for all $k\in\N$.
Denote $\hat x^k:=(x_1,\ldots,x^k,\ldots,x_n)$ ($x^k$ is in the $i$th position) for all $k\in\N$.
By \eqref{C4}, $\hat x^k \to\hat x$.
Thus, there is a sequence $\hat z^k:=(z_1^{k},\ldots,z_n^{k})\to\hat z$ such that, for all $k\in\N$, we have $\hat x^k+t_k\hat z^k\in\widehat\Omega$, and consequently, $x^k+t_kz_i^k\in\Omega_i$.
By \eqref{C3}, we have $z_i^k\to z_i$.
Hence, $z_i\in T^C_{\Omega_i}(x_i)$, and consequently, $T^C_{\widehat\Omega}(\hat x)\subset T^C_{\Omega_1}(x_1)\times\ldots\times T^C_{\Omega_n}(x_n)$.
		
Conversely, let $z_i\in T^C_{\Omega_i}(x_i)$ $(i=1,\ldots,n)$ and $\hat z:=(z_1,\ldots,z_n)$.
Take arbitrary sequences $t_k\downarrow 0$ and  $\hat x^k:=(x_1^{k},\ldots,x_n^{k})	\to \hat x$ with $\hat x^k\in\widehat\Omega$ for all $k\in\N$.
By \eqref{C3}, $x_i^k\to x_i$ for all $i=1,\ldots,n$.
Thus, there exist  $z_i^k\to z_i$ such that $x^k_{i}+t_kz^k_i\in\Omega_i$ for all $k\in\N$ and $i=1,\ldots,n$.
By \eqref{C4}, $(z_1^{k},\ldots,z_n^{k})\to\hat z$.
Hence, $\hat z\in T^C_{\widehat\Omega}(\hat x)$, and consequently, $T^C_{\Omega_1}(x_1)\times\ldots\times T^C_{\Omega_n}(x_n)\subset T^C_{\widehat\Omega}(\hat x)$.
\item
Let $\hat x^*:=(x^*_1,\ldots,x^*_n)\in N^C_{\widehat\Omega}(\hat x)$.
For each $i\in\{1,\ldots,n\}$ and $z\in T^C_{\Omega_i}(x_i)$, one has
$\hat z:=(0,\ldots,z,\ldots,0)\in T^C_{\widehat\Omega}(x)$
($z$ is in the $i$th position),
and consequently,
$\langle x^*_i,z\rangle=\langle\hat x^*,\hat z\rangle\le 0$.
Hence, $x^*_i\in N^C_{\Omega_i}(x_i)$, and consequently, $N^C_{\widehat\Omega}(\hat x)\subset N^C_{\Omega_1}(x_1)\times\ldots\times N^C_{\Omega_n}(x_n)$.

Conversely, let $x^*_i\in N^C_{\Omega_i}(x_i)$ $(i=1,\ldots,n)$, $\hat x^*:=(x^*_1,\ldots,x^*_n)$ and
$\hat z:=(z_1,\ldots,z_n)\in T^C_{\widehat\Omega}(\hat x)$.
By \eqref{P2.6-1}, $z_i\in T^C_{\Omega_i}(x_i)$ $(i=1,\ldots,n)$.
Hence, $\langle\hat x^*,\hat z\rangle= \sum_{i=1}^{n}\langle x^*_i,z_i\rangle\le 0$,
i.e., $\hat x^*\in N^C_{\widehat\Omega}(\hat x)$, and consequently,
$N^C_{\Omega_1}(x_1)\times\ldots\times N^C_{\Omega_n}(x_n)\subset N^C_{\widehat\Omega}(\hat x)$.

\item
Suppose that compatibility conditions \eqref{C3} and \eqref{C4} are satisfied with some $\kappa_1>0$ and $\kappa_2>0$, respectively.
Let $x^*_i\in N^F_{\Omega_i}(x_i)$ $(i=1,\ldots,n)$, and
 $\hat x^*:=(x^*_1,\ldots,x^*_n)$.
By \eqref{C3},
\begin{gather*}
\limsup_{\widehat\Omega\ni (x'_1,\ldots,x'_n)\to \hat x\atop (x'_1,\ldots,x'_n)\ne\hat x} \frac {\langle \hat x^*,(x'_1,\ldots,x'_n)-\hat x \rangle}{\vertiii{(x'_1,\ldots,x'_n)-\hat x}}
\le
{\kappa_1}
\sum_{i=1}^n\limsup_{\Omega_i\ni x'_i\to x_i, \; x'_i\ne x_i}
\frac{[\langle x^*_i,x'_i-x_i \rangle]_+}{\|x'_i-x_i\|}\le 0,
\end{gather*}
where we use the notation $\al_+:=\max\{\al,0\}$.
Thus, $\hat x^*\in N^F_{\widehat\Omega}(\hat x)$, and consequently,
$N^F_{\Omega_1}(x_1)\times\ldots\times N^F_{\Omega_n}(x_n)\subset N^F_{\widehat\Omega}(\hat x)$.

Conversely, let $\hat x^*:=(x^*_1,\ldots,x^*_n)\in N^F_{\widehat\Omega}(\hat x)$.
By \eqref{C4}, for each $i\in\{1,\ldots,n\}$,
\begin{gather*}
\limsup_{\Omega_i\ni x'_i\to x_i,\; x'_i\ne x_i}
\frac{\langle x^*_i,x'_i-x_i \rangle}{\|x'_i-x_i\|}
\le
{\kappa_2}\limsup_{\Omega_i\ni x'_i	\to x_i,\; x'_i\ne x_i}
\frac {[\langle\hat x^*,(0,\ldots,x'_i-x_i,\ldots,0) \rangle]_+}{\vertiii{(0,\ldots,x'_i-x_i,\ldots,0)}}\le 0.
\end{gather*}	
Thus, $x^*_i\in N^F_{\Omega_i}(x_i)$ $(i=1,\ldots,n)$, and consequently, $N^F_{\widehat\Omega}(\hat x)\subset N^F_{\Omega_1}(x_1)\times\ldots\times N^F_{\Omega_n}(x_n)$.
\end{enumerate}
The proof is complete.
\end{proof}

{
Conditions \eqref{C3} and \eqref{C4} are satisfied, e.g., for \emph{sign-symmetric} \cite{Cuong2} norms.
\begin{proposition}
Let $\|\cdot\|$ and $\vertiii{\;\cdot\;}$ be norms on $X$ and $X^n$, respectively.	
Suppose that they satisfy the following conditions:
\begin{enumerate}
\item
$\vertiii{(x_1,\ldots,x_n)}
=\vertiii{(\pm x_1,\ldots,\pm x_n)}$ for all $x_1,\ldots,x_n\in X$;
\item
$\vertiii{(x_1,\ldots,x_n)}
=\|x_i\|$ for each $i=1,\ldots,n$ and all $x_i\in X$ as long as $x_j=0$ if $j\ne i$.
\end{enumerate}
Then $\max\limits_{1\le i\le n}\|x_i\|\le \vertiii{(x_1,\ldots,x_n)}\le n\cdot \max\limits_{1\le i\le n}\|x_i\|$.
\end{proposition}
\begin{proof}
Let $x_1,\ldots,x_n\in X$.
For each $i=1,\ldots,n$, set $u_i=v_i=x_i$, $u_j=-x_j$ and $v_j=0$ if $j\ne i$.
Denote $\hat x:=(x_1,\ldots,x_n)$,
$\hat x_i^-:=(u_1,\ldots,u_n)$ and
$\hat x_i^0:=(v_1,\ldots,v_n)$.
Then $\hat x+\hat x_i^-=\hat x_i^0$ and
$\vertiii{\hat x_i^-}=\vertiii{\hat x}$.
Hence,
\begin{align*}
\|x_i\|
=\vertiii{\hat x_i^0} \le(\vertiii{\hat x} +\vertiii{\hat x_i^-})/2
=\vertiii{\hat x}.
\end{align*}	
Thus, $\max\limits_{1\le i\le n}\|x_i\|\le \vertiii{(x_1,\ldots,x_n)}$.
On the other hand,
\begin{align*}
\vertiii{\hat x}
\le\sum_{i=1}^{n}\vertiii{\hat x_i^0}
=\sum_{i=1}^{n}\|x_i\| \le n\cdot\max\limits_{1\le i\le n}\|x_i\|.
\end{align*}	
The proof is complete.
\end{proof}
}
\section{Generalized separation}
\label{S3}
We prove a generalized separation statement with arbitrary product norms covering all the results discussed in the Introduction.

\begin{theorem}
\label{T3.3}
Let $X$ be a vector space, $(X^{n-1},\vertiii{\,\cdot\,})$ and $(X^n,\vertiii{\,\cdot\,}_+)$ be a normed and a Banach spaces, respectively.
Assume that $\widehat\Omega$ is closed, $\widehat\omega:= (\omega_1,\ldots,\omega_n)\in\widehat\Omega$,
$\eps>0$ and $\de>0$.
Suppose that $\bigcap_{i=1}^n\Omega_i=\emptyset$ and
\begin{gather}
\label{T3.3-1}
\vertiii{(\omega_1-\omega_n,\ldots,\omega_{n-1}-\omega_{n})} <\inf_{u_i\in\Omega_i\;(i=1,\ldots,n)} \vertiii{(u_1-u_n,\ldots,u_{n-1}-u_{n})}+\varepsilon.
\end{gather}	
The following assertions hold true:
\begin{enumerate}
\item\label{T3.1-2}
there exist $\hat x:=(x_1,\ldots,x_n)\in\widehat\Omega\cap B_\de(\widehat\omega)$  and $\hat x^*:=(x_1^*,\ldots,x_n^*)\in (X^*)^n$ such that
\begin{gather}
\label{L1.8-2}
\sum_{i=1}^n x_i^*=0,\quad
\vertiii{(x_1^*,\ldots,x^*_{n-1})}=1,
\\
\label{T5.3-5}
d\left(\hat x^*,N^C_{\widehat\Omega}(\hat x)\right) <\dfrac{\eps}{\de},\\
\notag
\sum_{i=1}^{n-1} \langle x_i^*, x_n-x_i\rangle=\vertiii{(x_n-x_1,\ldots,x_n-x_{n-1})};
\end{gather}
\item\label{T3.1-3}
if $(X^n,\vertiii{\,\cdot\,}_+)$ is Asplund, and
the norms
$\vertiii{\,\cdot\,}$ and $\vertiii{\,\cdot\,}_+$ are compatible in the sense of \eqref{C1},
then, for any $\tau\in(0,1)$, there exist $\hat x:=(x_1,\ldots,x_n)\in\widehat\Omega\cap B_\de(\widehat\omega)$ and $\hat x^*:=(x_1^*,\ldots,x_n^*)\in (X^*)^n$ such that conditions \eqref{L1.8-2} are satisfied, and
\begin{gather}
\label{T5.3-3}
d\left(\hat x^*,N^F_{\widehat\Omega}(\hat x)\right)<\dfrac{\eps}{\de},\\
\label{T5.3-4}
\sum_{i=1}^{n-1} \langle x_i^*, x_n-x_i\rangle>\tau\vertiii{(x_n-x_1,\ldots,x_n-x_{n-1})}.
\end{gather}
\end{enumerate}
\end{theorem}

\begin{proof}
Define $f_1: X^{n}\rightarrow [0,+\infty)$ by
\begin{gather}
\label{f1}
f_1(\hat u):=\vertiii{(u_1-u_n,\ldots,u_{n-1}-u_n)}
\;\;
\text{for all}
\;\;
\hat u:=(u_1,\ldots,u_n)\in X^n.
\end{gather}
Let $\de>0$.
By \eqref{T3.3-1}, there exists an $\eps'\in(0,\eps)$
such that
\begin{gather}
\label{T4.1P2}
f_1(\widehat\omega) <\inf_{u\in\widehat\Omega} f_1(\hat u)+\eps'.
\end{gather}
The proof is divided into a series of claims.
\smallskip

\underline{Claim 1}.
There exists a point
$\hat x^\circ:=(x^\circ_1,\ldots,x^\circ_n)\in\widehat\Omega$ such that
\begin{gather}
\label{T5.2-2}
\vertiii{\hat x^\circ-\widehat\omega}_+<\de,
\\
\label{T5.2-3}
f_1(\hat u)-f_1(\hat x^\circ)+
\frac{\eps'}{\de}
\vertiii{\hat u-\hat x^\circ}_+\geq 0
\;\;
\text{for all}
\;\;
\hat u\in\widehat\Omega,
\\
\label{T5.2-3b}
\vertiii{(x^\circ_1-x^\circ_n,\ldots,x^\circ_{n-1}-x^\circ_n)}>0.
\end{gather}

The closed set $\widehat\Omega$ equipped with the distance induced by the norm $\vertiii{\,\cdot\,}_+$ is a complete metric space.
The function $f_1$ is continuous.
Conditions \eqref{T5.2-2} and \eqref{T5.2-3} are consequences of \eqref{T4.1P2} thanks to the \EVP\ (Lem\-ma~\ref{EVP}) applied to the restriction of $f_1$ to $\widehat\Omega$.
Since $\bigcap_{i=1}^n\Omega_i=\emptyset$, condition \eqref{T5.2-3b} is satisfied.
\xqed

As an immediate consequence of \eqref{T5.2-3}, we have
\smallskip

\underline{Claim 2}.
$0\in\sd^F(f_1+f_2+f_3)(\hat x^\circ)$, where $f_1$ is given by \eqref{f1} and
\begin{gather*}
\notag
f_2(\hat u):=\frac{\eps'}{\de}\vertiii{\hat u-\hat x^\circ}_+,\quad
f_3(\hat u):=i_{\widehat\Omega}(\hat u)
\;\;
\text{for all}
\;\;
\hat u\in X^n.
\end{gather*}

Indeed, it follows from \eqref{T5.2-3} that $\hat x^\circ$ is a global minimum of $f_1+f_2+f_3$.
\xqed

Note that $f_1$ and $f_2$ are convex and Lipschitz continuous.
The first two items in the next claim are standard facts of convex analysis, while the last one is a consequence of the definitions.
\smallskip

\underline{Claim 3}.
For all $\hat u:=(u_1,\ldots,u_n)\in X^n$, the following relations hold:
	
{\renewcommand{\theenumi}{\alph{enumi}}
\begin{enumerate}
\item\label{a}
if $(u_{1}^*,\ldots,u_{n}^*) \in \partial f_1(\hat u)$ and
$\vertiii{(u_1-u_n,\ldots,u_{n-1}-u_n)}>0$, then
$\sum_{i=1}^n u_{i}^*=0$, $\vertiii{(u^*_1,\ldots,u^*_{n-1})}=1$
and $\sum_{i=1}^{n-1}\langle u_{i}^*,u_i-u_n\rangle= \vertiii{(u_1-u_n,\ldots,u_{n-1}-u_n)}$;
\item\label{b}
if $\hat u^*\in \partial f_2(\hat u)$, then
$\vertiii{\hat u^*}_+\leq {\eps'}/{\de}$;
\item\label{c}
if $\hat u\in\widehat\Omega$, then $\partial^C f_3(\hat u)=N^C_{\widehat\Omega}(\hat u)$ and $\partial^F f_3(\hat u)= N^F_{\widehat\Omega}(\hat u)$.
\end{enumerate}
}

The next step is to apply a subdifferential sum rule.
\smallskip

\underline{Claim 4}.
Assertion \eqref{T3.1-2} holds true.
\smallskip

Claim 2 implies $0\in\sd^C(f_1+f_2+f_3)(\hat x^\circ)$.
Applying Clarke subdifferential sum rule (Lem\-ma~\ref{SR}\,\eqref{SR.Cl}), we can find subgradients
$\hat y_{1}^*:=(y_{1}^*,\ldots,y_{n}^*) \in \partial f_1(\hat x^\circ)$, $\hat y_{2}^*\in\partial f_2(\hat x^\circ)$ and $\hat y_{3}^*\in\partial^C f_3(\hat x^\circ)$ such that
\begin{gather}
\label{T3.2-12}
\hat y_{1}^*+\hat y_{2}^*+\hat y_{3}^*=0.
\end{gather}
Thanks to \eqref{T5.2-3b},
it follows from the relations in Claim 3 that $\hat y_{3}^*\in N^C_{\widehat\Omega}(\hat x^\circ)$, and
\begin{gather}\label{T3.3-12}
\sum_{i=1}^n y_{i}^*=0, \quad \vertiii{(y_{1}^*,\ldots,y_{n-1}^*)}=1,\\
\notag
\langle (y_{1}^*,\ldots,y_{n-1}^*),(x^\circ_1-x^\circ_n,\ldots,x^\circ_{n-1}-x^\circ_n)\rangle=\vertiii{(x^\circ_1-x^\circ_n,\ldots,x^\circ_{n-1}-x^\circ_n)},\\
\label{T3.3-14}
\vertiii{\hat y^*_{2}}_+\leq \dfrac{\eps'}{\de}.
\end{gather}
Let  $\hat x:=(x_1,\ldots,x_n):=(x^\circ_1,\ldots,x^\circ_n)$ and $\hat x^*:=(x_1^*,\ldots,x_n^*):=-
{\hat y_1^*}$.
By \eqref{T3.2-12}, \eqref{T3.3-12} and \eqref{T3.3-14},
\begin{align*}
d\left(\hat x^*,N^C_{\widehat\Omega}(\hat x)\right)
&\le\vertiii{\hat y_{1}^*+\hat y_{3}^*}_+=\vertiii{\hat y^*_{2}}_+\leq \dfrac{\eps'}{\de}<\dfrac{\eps}{\de}.
\end{align*}
Thus, all the conditions in assertion \eqref{T3.1-2} are satisfied.
\xqed

Next, we prove assertion \eqref{T3.1-3}.
Suppose that $(X^n,\vertiii{\,\cdot\,}_+)$ is Asplund, and
 $\vertiii{\,\cdot\,}$ and $\vertiii{\,\cdot\,}_+$ are compatible in the sense of \eqref{C1} with some $\kappa>0$.
Let $\tau\in(0,1)$.
In view of \eqref{T5.2-2} and \eqref{T5.2-3b}, we can
choose a $\xi>0$ so that
\begin{gather}
\label{T3.2-1}
\xi<\max\left\{\de-\vertiii{\hat x^\circ-\widehat\omega}_+, (\varepsilon-\varepsilon')/\de\right\},
\\
\label{T3.2-5}
\text{and}\;\;
10\kappa\xi<(1-\tau)\vertiii{(x^\circ_1-x^\circ_n,\ldots,x^\circ_{n-1}-x^\circ_n)}.
\end{gather}
The next claim is a consequence of Claim 2, the fuzzy \Fr\ subdifferential sum rule combined with the convex sum rule (Lem\-ma~\ref{SR}\,\eqref{SR.convex} and \eqref{SR.Fr}) and relation \eqref{c} in Claim 3.
\smallskip

\underline{Claim 5}.
There exist points $\hat x:=(x_1,\ldots,x_n)\in\widehat\Omega$, $\hat x':=(x'_1,\ldots,x'_n)\in {X^{n}}$,
$\hat y_{1}^*\in \partial f_1(\hat x')$,
$\hat y_{2}^*\in\partial f_2(\hat x')$ and
$\hat y_{3}^*\in N^F_{\widehat\Omega}(\hat x)$
such that
\begin{gather}
\label{T3.2-2}
\vertiii{\hat x-\hat x^\circ}_+<\xi,\quad \vertiii{\hat x'-\hat x^\circ}_+<\xi,\\
\label{T3.3-3}
\vertiii{\hat y_{1}^*+\hat y_{2}^*+\hat y_{3}^*}_+<\xi.
\end{gather}

\underline{Claim 6}.
$\hat x\in B_\de(\widehat\omega)$, conditions \eqref{L1.8-2} and \eqref{T5.3-3} are satisfied with $\hat x^*:=(x_1^*,\ldots,x_n^*):=-\hat y_{1}^*$, and
\begin{gather}
\label{T3.3-15}
\sum_{i=1}^{n-1}\langle x^*_{i},x'_n-x'_i\rangle =\vertiii{(x'_n-x'_1,\ldots,x'_n-x'_{n-1})}.
\end{gather}

Indeed,
by \eqref{T3.2-1} and \eqref{T3.2-2},
$
\vertiii{\hat x-\widehat\omega}_+
\le\vertiii{\hat x^\circ-\widehat\omega}_++\vertiii{\hat x-\hat x^\circ}_+<\vertiii{\hat x^\circ-\widehat\omega}_++\xi<\de.
$
By \eqref{C1} and \eqref{T3.2-2}, we have
\begin{align*}
\notag
\vertiii{(x'_1-x'_n,\ldots,x'_{n-1}-x'_n)}
&\ge \vertiii{(x^\circ_1-x^\circ_n,\ldots,x^\circ_{n-1}-x^\circ_n)}
\\ \notag
&-\vertiii{(x'_1- x^\circ_1,\ldots,x'_{n-1}-x^\circ_{n-1})} - \vertiii{(x'_n- x^\circ_n,\ldots,x'_{n}-x^\circ_n)_{n-1}}\\
&>\vertiii{(x^\circ_1-x^\circ_n,\ldots,x^\circ_{n-1}-x^\circ_n)}-2\kappa\xi>0,
\end{align*}
and relation \eqref{a} in Claim 3 implies that conditions \eqref{L1.8-2} and \eqref{T3.3-15} are satisfied, while relation \eqref{b} yields \eqref{T3.3-14}.
By \eqref{T3.3-14}, \eqref{T3.2-1} and \eqref{T3.3-3},
\begin{align*}
d\left(\hat x^*,N^F_{\widehat\Omega}(\hat x)\right)
&\le\vertiii{\hat y_{1}^*+\hat y_{3}^*}_+<\vertiii{\hat y_{2}^*}_++\xi\le \dfrac{\eps'}{\de}+\xi<\dfrac{\eps}{\de}.
\end{align*}
This proves \eqref{T5.3-3}.
\xqed

\underline{Claim 7}.
Condition \eqref{T5.3-4} is satisfied.
\smallskip

By \eqref{C1} and \eqref{T3.2-2}, we have
\begin{align}
\label{T3.2-13a}
\vertiii{(x'_1-x_1,\ldots,x'_{n-1}-x_{n-1})} &\le\kappa\vertiii{\hat x'-\hat x}_+\le\kappa(\vertiii{\hat x'-\hat x^\circ}_++\vertiii{\hat x-\hat x^\circ}_+)<2\kappa\xi,
\\
\label{T3.2-13b}
\vertiii{(x'_n-x_n,\ldots,x'_{n}-x_{n})_{n-1}} &\le\kappa\vertiii{\hat x'-\hat x}_+<2\kappa\xi,
\\
\label{T3.3-11a}
\vertiii{(x_1-x_n,\ldots,x_{n-1}-x_n)}&> \vertiii{(x^\circ_1-x^\circ_n,\ldots,x^\circ_{n-1}-x^\circ_n)}-2\kappa\xi,
\\
\label{T3.3-11}
\vertiii{(x'_1-x'_n,\ldots,x'_{n-1}-x'_n)}
&>\vertiii{(x_1-x_n,\ldots,x_{n-1}- x_n)}-4\kappa\xi.
\end{align}
By \eqref{T3.2-5} and \eqref{T3.3-11a},
\begin{align}
\label{T3.2-9}
(1-\tau)\vertiii{(x_1-x_n,\ldots,x_{n-1}-x_n)}
&>10\kappa\xi-2\kappa(1-\tau)\xi>8\kappa\xi.
\end{align}
Using  \eqref{L1.8-2}, \eqref{T3.3-15}, \eqref{T3.2-13a}, \eqref{T3.2-13b}, \eqref{T3.3-11} and \eqref{T3.2-9}, we obtain
\begin{align*}
\sum_{i=1}^{n-1} \langle x_i^*,x_n-x_i\rangle
=& \sum_{i=1}^{n-1}\langle x^*_{i},x'_n-x'_i\rangle+\sum_{i=1}^{n-1}\langle x^*_{i},x'_i-x_i\rangle
+\sum_{i=1}^{n-1}\langle x^*_{i},x_n-x'_n\rangle\\
\overset{\eqref{L1.8-2},\eqref{T3.3-15}}{\ge}& \vertiii{(x'_1-x'_n,\ldots,x'_{n-1}-x'_n)}\\
-&\vertiii{(x'_1-x_1,\ldots,x'_{n-1}-x_{n-1})} -\vertiii{(x'_n-x_n,\ldots,x'_n-x_n)_{n-1}}\\
\overset{\eqref{T3.2-13a}, \eqref{T3.2-13b}}{>} &\vertiii{(x'_1-x'_n,\ldots,x'_{n-1}-x'_n)}-4\kappa\xi\\
\overset{\eqref{T3.3-11}}{>}& \vertiii{(x_1-x_n,\ldots,x_{n-1}-x_n)}-8\kappa\xi\\
\overset{\eqref{T3.2-9}}{>}& \tau\vertiii{(x_1-x_n,\ldots,x_{n-1}-x_n)}.
\end{align*}
This proves \eqref{T5.3-4}.
\end{proof}

\begin{remark}\label{R3}
Suppose $(X,\|\cdot\|)$ is a normed space.
\begin{enumerate}
\item
Conditions \eqref{T5.3-5} and \eqref{T5.3-3} in the conclusions of Theorem~\ref{T3.3} are formulated in terms of normal cones to the aggregate set $\widehat\Omega$.
They can be reformulated in terms of normal cones to the individual sets when $\|\cdot\|$ and $\vertiii{\,\cdot\,}_+$
are compatible in the sense of \eqref{C3} and \eqref{C4}; see Proposition~\ref{P2.3}.
\item
Theorem~\ref{ZhNg} is a particular case of Theorem~\ref{T3.3} with maximum norms of the type \eqref{norm}.
Theorem~\ref{T1.4} is a particular case of Theorem~\ref{T3.3} with $\vertiii{\,\cdot\,}$ of the type \eqref{norm} and $\vertiii{\,\cdot\,}_+$ of the type \eqref{ganorm}:
$$\vertiii{(u_1,\ldots,u_n)}_+:= \max\Big\{\vertiii{(u_1,\ldots,u_{n-1})}, \frac\de\eta\|u_n\|\Big\}
\;\;
\text{for all}
\;\;
u_1,\ldots,u_{n}\in X.$$
The general $p$-weighted ($p\in[1,+\infty]$) version of Theorem~\ref{ZhNg} in \cite[Theorems 3.1 \& 3.4]{ZheNg11} (cf. Remark~\ref{R1}\,\eqref{R1.1}) is a particular case of Theorem~\ref{T3.3} with
$$
\vertiii{(u_1,\ldots,u_{n-1})}:= \left(\sum_{i=1}^{n-1}\|u_i\|^p\right)^{\frac{1}{p}}
\AND
\vertiii{(u_1,\ldots,u_n)}_+:= \left(\sum_{i=1}^{n}\|u_i\|^p\right)^{\frac{1}{p}}.
$$
A more general multi-parameter statement covering both Theorems~\ref{ZhNg} and \ref{T1.4} can be deduced from Theorem~\ref{T3.3} using
$$
\vertiii{(u_1,\ldots,u_{n-1})}:= \Big(\sum_{i=1}^{n-1} (\sigma_i\|u_i\|)^p\Big)^{\frac{1}{p}}
\;\;\text{and}\;\;
\vertiii{(u_1,\ldots,u_n)}_+:= \Big(\sum_{i=1}^{n} (\sigma_i\|u_i\|)^p\Big)^{\frac{1}{p}},
$$
where $\sigma_1,\ldots,\sigma_n$ are given positive numbers
(suggested by the anonymous reviewer of the paper \cite{CuoKruTha24}).
\end{enumerate}
\end{remark}


The next local version of Theorem~\ref{T3.3} is a direct consequence of this theorem.

\begin{corollary}
\label{C3.3}
Let $(X,\|\cdot\|)$ and $(X^n,\vertiii{\,\cdot\,})$ be Banach.
Assume that
$\widehat\Omega$ is closed, $\widehat\omega:= (\omega_1,\ldots,\omega_n)\in\widehat\Omega$, $\bx\in X$, $\eps>0$, $\de>0$ and $\rho>0$.
Suppose that $\bigcap_{i=1}^n\Omega_i\cap\overline B_\rho(\bx)=\emptyset$, and
\begin{gather*}
\vertiii{(\omega_1-\bx,\ldots,\omega_{n}-\bx)} <\inf_{\substack{u_i\in\Omega_i\;(i=1,\ldots,n),\; u\in B_\rho(\bx)}} \vertiii{(u_1-u,\ldots,u_{n}-u)}+\varepsilon.
\end{gather*}	
\if{
\NDC{30/11/24.
I think the ball $B_\rho(\bx)$ in the RHS of the above inequality should be replaced by $\overline B_\rho(\bx)$.
Similarly, $\B$ in the RHS of condition \eqref{C3.9-3} should be replaced by $\overline\B$.
}
\AK{3/12/24.
This does not affect the value of the infimum.}
}\fi
The following assertions hold true:
\begin{enumerate}
\item
there exist $\hat x:=(x_1,\ldots,x_n)\in\widehat\Omega\cap B_\de(\widehat\omega)$, {$x_0\in B_\rho(\bx)$} and $\hat x^*:=(x_1^*,\ldots,x_n^*)\in (X^*)^n$ such that $\vertiii{\hat x^*}=1$, and
\begin{gather}
\label{C3.3-4}
{\de}d\left(\hat x^*,N^C_{\widehat\Omega}(\hat x)\right)+ \rho\Big\|\sum_{i=1}^nx_i^*\Big\| <{\eps},\\
\notag
\sum_{i=1}^{n} \langle x_i^*, x_0-x_i\rangle=\vertiii{(x_0-x_1,\ldots,x_0-x_{n})};
\end{gather}
\item
if $(X,\|\cdot\|)$ and $(X^n,\vertiii{\,\cdot\,})$ are Asplund, and the norms $\|\cdot\|$ and $\vertiii{\,\cdot\,}$  are compatible in the sense of \eqref{C5},
then, for any $\tau\in(0,1)$, there exist $\hat x:=(x_1,\ldots,x_n)\in\widehat\Omega\cap B_\de(\widehat\omega)$, {$x_0\in B_\rho(\bx)$} and $\hat x^*:=(x_1^*,\ldots,x_n^*)\in (X^*)^n$ such that $\vertiii{\hat x^*}=1$, and
\begin{gather}
\label{C3.3-6}
{\de}d\left(\hat x^*,N^F_{\widehat\Omega}(\hat x)\right)+ \rho\Big\|\sum_{i=1}^nx_i^*\Big\| <{\eps},\\
\notag
\sum_{i=1}^{n} \langle x_i^*, x_0-x_i\rangle>\tau\vertiii{(x_0-x_1,\ldots,x_0-x_{n})}.
\end{gather}
\end{enumerate}
\end{corollary}

\begin{proof}
Define the norm $\vertiii{\,\cdot\,}_+$ on $X^{n+1}$ by
\begin{gather*}
\vertiii{(u_1,\ldots,u_{n+1})}_+:= \max\Big\{\vertiii{(u_1,\ldots,u_{n})}, \frac\de\rho\|u_{n+1}\|\Big\}
\;\;
\text{for all}
\;\;
u_1,\ldots,u_{n+1}\in X,
\end{gather*}
and observe that $(X^{n+1},\vertiii{\,\cdot\,}_+)$ is Banach or Asplund whenever $(X,\|\cdot\|)$ and $(X^n,\vertiii{\,\cdot\,})$ are.
The corresponding dual norm is given by
\begin{gather*}
\vertiii{(u^*_1,\ldots,u^*_{n+1})}_+= \vertiii{(u^*_1,\ldots,u^*_n)}+\frac\rho\de\|u^*_{n+1}\|
\;\;
\text{for all}
\;\;
u^*_1,\ldots,u^*_{n+1}\in X^*.
\end{gather*}
Moreover, if $\|\cdot\|$ and $\vertiii{\,\cdot\,}$ satisfy condition \eqref{C5} with some $\kappa>0$, then
$
\vertiii{(u,\ldots,u)_n}\le
\kappa\|u\|
=\frac{\kappa\rho}\de\frac\de\rho\|u\|
$
for all $u\in X$, and consequently,  $\vertiii{\,\cdot\,}$ and $\vertiii{\,\cdot\,}_+$ are compatible in the sense of \eqref{C1} with $\kappa':=\max\{1,\frac{\kappa\rho}\de\}$.
Set $\Omega_{n+1}:=\overline B_\rho(\bx)$ and $\omega_{n+1}:=\bx$, and apply Theorem~\ref{T3.3} with $n+1$ in place of $n$.
Observe that
$N^C_{\Omega_{n+1}}(u)=N^F_{\Omega_{n+1}}(u)=\{0\}$ for all $u\in B_\rho(\bx)$.
\end{proof}

For what follows, we need a  more general setting involving a collection of points $x^\circ_1,\ldots,x^\circ_n\in X$ instead of a single point $\bx$.
The next corollary is a direct consequence of Corollary~\ref{C3.3} applied to the
sets $\widehat\Omega_i:=\Omega_i-x^\circ_i$, and points $\widehat\omega_i:=\omega_i-x^\circ_i\in\widehat\Omega_i$ $(i=1,\ldots,n)$ and $\bx:=0$.	

\begin{corollary}
\label{C3.9}
Let $(X,\|\cdot\|)$ and $(X^n,\vertiii{\,\cdot\,})$ be Banach.
Assume that
$\widehat\Omega$ is closed, $\widehat\omega:= (\omega_1,\ldots,\omega_n)\in\widehat\Omega$, $x^\circ_i\in X$ $(i=1,\ldots,n)$, $\eps>0$, $\de>0$ and $\rho>0$.
Suppose that  ${\bigcap}_{i=1}^n(\Omega_i-x^\circ_i)\cap (\rho\overline\B)=\emptyset$, and
\begin{gather}
\label{C3.9-3}
\vertiii{(\omega_1-x^\circ_1,\ldots,\omega_{n}-x^\circ_n)} <\inf_{\substack{u_i\in\Omega_i\;(i=1,\ldots,n),\; u\in\rho\B}} \vertiii{(u_1-x^\circ_1-u,\ldots,u_{n}-x^\circ_n-u)}+\varepsilon.
\end{gather}	
The following assertions hold true:
\begin{enumerate}
\item\label{C3.3-1}
there exist $\hat x:=(x_1,\ldots,x_n)\in\widehat\Omega\cap B_\de(\widehat\omega)$, {$x_0\in\rho\B$} and $\hat x^*:=(x_1^*,\ldots,x_n^*)\in (X^*)^n$ such that $\vertiii{\hat x^*}=1$,
condition \eqref{C3.3-4} is satisfied, and
\begin{gather}
\notag
\sum_{i=1}^{n} \langle x_i^*, x_0+x^\circ_i-x_i\rangle= \vertiii{(x_0+x^\circ_1-x_1,\ldots,x_0+x^\circ_n-x_{n})};
\end{gather}
\item\label{C3.3-2}
if $(X,\|\cdot\|)$ and $(X^n,\vertiii{\,\cdot\,})$ are Asplund, and the norms $\|\cdot\|$ and $\vertiii{\,\cdot\,}$ are compatible in the sense of \eqref{C5},
then, for any $\tau\in(0,1)$, there exist $\hat x:=(x_1,\ldots,x_n)\in\widehat\Omega\cap B_\de(\widehat\omega)$, {$x_0\in\rho\B$} and $\hat x^*:=(x_1^*,\ldots,x_n^*)\in (X^*)^n$ such that $\vertiii{\hat x^*}=1$,
condition \eqref{C3.3-6} is satisfied, and
\begin{gather}
\notag
\sum_{i=1}^{n} \langle x_i^*, x_0+x^\circ_i-x_i\rangle>\tau \vertiii{(x_0+x^\circ_1-x_1,\ldots,x_0+x^\circ_n-x_{n})}.
\end{gather}
\end{enumerate}
\end{corollary}

\begin{remark}
\label{R4}
\begin{enumerate}
\item\label{R4-1}
Corollary~\ref{C3.3} is a particular case of Corollary~\ref{C3.9} with $x^\circ_1=\ldots=x^\circ_n:=\bx$.
\item\label{R4-2}
When applying Corollary~\ref{C3.9}, it can be convenient to replace \eqref{C3.9-3} by a simpler (but slightly stronger) condition $\vertiii{(\omega_1-x^\circ_1,\ldots,\omega_n-x^\circ_n)} <\varepsilon$.
\end{enumerate}
\end{remark}

\section{Approximate stationarity and transversality}
\label{S4}

As an application of the generalized separation results in Section~\ref{S3}, we prove dual necessary conditions for an abstract product norm extension of the approximate stationarity
from \cite{Kru98,Kru06,Kru09,BuiKru19,CuoKru21.2}.
The next definition extends \cite[Definitions 2.1\,(iv) and 3.1]{BuiKru19}.

\begin{definition}
\label{D3.5}
Let $(X,\|\cdot\|)$ and $(X^n,\vertiii{\,\cdot\,})$ be normed spaces, and $\bx\in\bigcap_{i=1}^n\Omega_i$.
\begin{enumerate}
\item\label{D4.1-1}
Let $\al>0$.
The collection $\{\Omega_1,\ldots,\Omega_n\}$ is approximately $\al$-stationary at $\bx$ if, for any $\varepsilon>0$, there exist $\rho\in(0,\varepsilon)$,  $x_i\in\Omega_i$ and  $a_i\in X$ $(i=1,\ldots,n)$ such that $\vertiii{(x_1-\bx,\ldots,x_n-\bx)}<\varepsilon$,
$\vertiii{(a_1,\ldots,a_n)}<\al\rho$, and
$\bigcap_{i=1}^n(\Omega_i-x_i-a_i)\cap (\rho\B)=\emptyset$.
\item\label{D4.1-2}
The collection $\{\Omega_1,\ldots,\Omega_n\}$ is approximately stationary at $\bx$ if it is approximately $\al$-sta\-tionary at $\bx$ for all $\al>0$.
\end{enumerate}
\end{definition}	

The next theorem establishes dual necessary conditions for approximate $\al$-statio\-narity.

\begin{theorem}
\label{T4.3}
Let $(X,\|\cdot\|)$ and $(X^n,\vertiii{\,\cdot\,})$ be Banach, and compatibility condition \eqref{C6} be satisfied.
Assume that $\widehat\Omega$ is closed, $\bx\in\bigcap_{i=1}^n\Omega_i$, and $\al>0$.
Suppose that $\{\Omega_1,\ldots,\Omega_n\}$ is approximately $\al$-sta\-tionary at $\bx$.
The following assertions hold true:
\begin{enumerate}
\item\label{T4.2-1}
for any $\varepsilon>0$, $\be>\al$ and $\tau\in(0,1)$, there exist $\hat x:=(x_1,\ldots,x_n), (x'_1,\ldots,x'_n)\in\widehat\Omega$, $\hat a:=(a_1,\ldots,a_n)\in X^n$, {$x_0\in X$} and $\hat x^*:=(x_1^*,\ldots,x_n^*)\in N^C_{\widehat\Omega}(\hat x)$ such that
\begin{gather}
\label{T4.3-3}
\hspace{-6mm}
\vertiii{(x_1-\bx,\ldots,x_n-\bx)}<\varepsilon,\;\;
\vertiii{(x'_1-\bx,\ldots,x'_n-\bx)}<\varepsilon,\;\;
\vertiii{\hat a}<\varepsilon,\;\;
\|x_0\|<\varepsilon,\\
\label{T4.3-5}
\Big\|{\sum_{i=1}^nx_i^*}\Big\|<\be,\quad
\vertiii{\hat x^*}=1,\\
\label{T4.3-2}
\hspace{-6mm}
\sum_{i=1}^{n}\langle x_i^*,x_0+a_i+x'_i-x_i\rangle >\tau\vertiii{(x_0+a_1+x'_1-x_1,\ldots,x_0+a_n+x'_n-x_n)};
\end{gather}	
\item\label{T4.2-2}
if $(X,\|\cdot\|)$ and $(X^n,\vertiii{\,\cdot\,})$ are Asplund, and  compatibility condition \eqref{C5} is satisfied,
then $N^C$ in \eqref{T4.2-1} can be replaced by $N^F$.
\end{enumerate}
\end{theorem}

\begin{proof}
Let $\eps>0$, $\be>\al$, $\tau\in(0,1)$, and  compatibility condition \eqref{C6} be satisfied with some $\kappa>0$. Choose a number $\xi>0$ so that
\begin{gather}
\label{T4.3P3}
2\xi<1-\tau,\quad \al\xi+\xi^2<\eps\AND \frac{\al+\kappa\xi}{1-\xi}<\be.
\end{gather}
By Definition~\ref{D3.5}\,\eqref{D4.1-1}, there exist a $\rho\in(0,\xi^2)$, $\hat x':=(x'_1,\ldots,x'_n)\in\widehat\Omega$ and $\hat a:=(a_1,\ldots,a_n)\in X^n$ such that $\vertiii{(x'_1-\bx,\ldots,x'_n-\bx)}<\xi^2$,
$\vertiii{\hat a}<\al\rho$, and $\bigcap_{i=1}^n(\Omega_i-x'_i-a_i)\cap (\rho\B)=\emptyset$.
{Taking a slightly smaller $\rho$, we can ensure that
$\bigcap_{i=1}^n(\Omega_i-x'_i-a_i)\cap (\rho\overline\B)=\emptyset$.}
Set $\eps':=\al\rho$ and $\de:=\al\rho^{\frac12}$.
By Corollary~\ref{C3.9}\,\eqref{C3.3-1} (see also Remark~\ref{R4}\,\eqref{R4-2}),
there exist
$\hat x:=(x_1,\ldots,x_n)\in\widehat\Omega\cap B_\de(\hat x')$, {$x_0\in\rho\B$} and $\hat x'^*:=(x_1'^*,\ldots,x_n'^*)\in(X^*)^n$ such that $\vertiii{\hat x'^*}=1$ and
\begin{gather}
\label{T4.3-4}
\de d\left(\hat x'^*,N^C_{\widehat\Omega}(\hat x)\right)+ \rho\Big\|\sum_{i=1}^nx_i'^*\Big\|<\eps',\\
\label{T4.3-9}
\sum_{i=1}^{n}\langle x_i'^*,x_0+a_i+x'_i-x_i\rangle=m,
\end{gather}
where $m:=\vertiii{(x_0+a_1+x'_1-x_1,\ldots,x_0+a_n+x'_n-x_n)}$.
In view of \eqref{T4.3P3},
\begin{gather*}
\vertiii{(x'_1-\bx,\ldots,x'_n-\bx)}<\xi^2<\varepsilon,\;\;
\vertiii{\hat a}<\al\rho<\al\xi^2<\al\xi<\varepsilon,\;\;
\|x_0\|<\rho<\xi^2<\varepsilon,\\
\vertiii{(x_1-\bx,\ldots,x_n-\bx)}\le 	
\vertiii{\hat x-\hat x'}+
\vertiii{(x'_1-\bx,\ldots,x'_n-\bx)}<\de+\xi^2< \al\xi+\xi^2<\varepsilon.
\end{gather*}	
This proves \eqref{T4.3-3}.
By \eqref{T4.3-4},
\begin{gather*}
d\left(\hat x'^*,N^C_{\widehat\Omega}(\hat x)\right) <\frac{\eps'}{\de}= \frac{\al\rho}{\al\rho^{\frac12}}<\xi,\quad \Big\|\sum_{i=1}^nx_i'^*\Big\|<\frac{\eps'}{\rho}=\al.
\end{gather*}	
Thus, there is a $\hat z^*:=(z_1^*,\ldots,z_n^*)\in N^C_{\widehat\Omega}(\hat x)$ such that $\vertiii{\hat x'^*-\hat z^*}<\xi$, and consequently, ${0<1-\xi}<\vertiii{\hat z^*}<1+\xi$.
By \eqref{C6},
\begin{gather*}
\Big\|\sum_{i=1}^nz_i^*\Big\|\le \Big\|\sum_{i=1}^n x_i'^*\Big\|+
\sum_{i=1}^{n}\|z^*_i-x_i'^*\|<\al+\kappa\xi.
\end{gather*}
Set $\hat x^*:=(x_1^*,\ldots,x_n^*):=\dfrac{\hat z^*}{\vertiii{\hat z^*}}$.
Then $\hat x^*\in N^C_{\widehat\Omega}(\hat x)$, and
\begin{gather*}
\vertiii{\hat x^*}=1,\quad
\Big\|\sum_{i=1}^nx_i^*\Big\| <\frac{\al+\kappa\xi}{1-\xi}<\be.
\end{gather*}
Hence, conditions \eqref{T4.3-5} are satisfied.
Moreover,
\begin{align*}
\vertiii{\hat x^*-\hat x'^*}
&\le \vertiii{\frac{\hat z^*}{\vertiii{\hat z^*}^*}-\hat z^*}+ \vertiii{\hat z^*-\hat x'^*}<\abs{\vertiii{\hat z^*}^*-1}+\xi<2\xi.
\end{align*}
Condition \eqref{T4.3-2} is a consequence of \eqref{T4.3-9}:
\begin{align}
\label{T4.3-6}
\sum_{i=1}^n\langle x_i^*,x_0+a_i+x'_i-x_i\rangle
&>\sum_{i=1}^n\langle x_i'^*, x_0+a_i+x'_i-x_i\rangle- 2\xi m=(1-2\xi)m>\tau m.
\end{align}
	
Suppose $(X,\|\cdot\|)$ and $(X^n,\vertiii{\,\cdot\,})$ are Asplund, and condition \eqref{C5} is satisfied.
Let $\hat \tau\in(\tau+2\xi,1)$.
Application of Corollary~\ref{C3.9}\;\eqref{C3.3-2} with $\hat \tau$ in place of $\tau$ in the above proof justifies conditions \eqref{T4.3-5} with $\hat x^*\in N^F_{\widehat\Omega}( \hat x)$, while the factor $1-2\xi$ in \eqref{T4.3-6} needs to be replaced by $\hat\tau-2\xi$ leading to the same  estimate.
This again proves \eqref{T4.3-2}.	
\end{proof}

\begin{corollary}
\label{C4.3}
Let $(X,\|\cdot\|)$ and $(X^n,\vertiii{\,\cdot\,})$ be Banach, and compatibility condition \eqref{C6} be satisfied.
Assume that $\widehat\Omega$ is closed, and $\bx\in\bigcap_{i=1}^n\Omega_i$.
Suppose that $\{\Omega_1,\ldots,\Omega_n\}$ is approximately sta\-tionary at $\bx$.
The following assertions hold true:
\begin{enumerate}
\item\label{C4.3-2}
for any $\varepsilon>0$ and $\tau\in(0,1)$, there exist $\hat x:=(x_1,\ldots,x_n), (x'_1,\ldots,x'_n)\in\widehat\Omega$, $\hat a:=(a_1,\ldots,a_n)\in X^n$, {$x_0\in X$} and $\hat x^*:=(x_1^*,\ldots,x_n^*)\in N^C_{\widehat\Omega}(\hat x)$ such that conditions \eqref{T4.3-3} and \eqref{T4.3-2} are satisfied, and
\begin{gather}
\label{C4.3-1}
\Big\|\sum_{i=1}^nx_i^*\Big\|<\eps,\quad
\vertiii{\hat x^*}=1;
\end{gather}	
\item\label{C4.3-3}
if $(X,\|\cdot\|)$ and $(X^n,\vertiii{\,\cdot\,})$ are Asplund, and compatibility condition \eqref{C5} is satisfied,
then $N^C$ in \eqref{C4.3-2} can be replaced by $N^F$.
\end{enumerate}
\end{corollary}

\begin{remark}
\label{R05}
\begin{enumerate}
\item
\label{R05.1}
Dropping inequality \eqref{T4.3-2} in Theorem~\ref{T4.3}\,\eqref{T4.2-1} together with the variables involved only in this condition leads to the simplified necessary conditions in the next assertion which can be employed in Theorem~\ref{T4.3} instead of assertion \eqref{T4.2-1}.
\begin{enumerate}
\item[(i)$^\prime$]\label{R5-1}
for any $\varepsilon>0$ and $\be>\al$, there exist $\hat x:=(x_1,\ldots,x_n)\in\widehat\Omega$ and $\hat x^*:=(x_1^*,\ldots,x_n^*)\in N^C_{\widehat\Omega}(\hat x)$ such that
$\vertiii{(x_1-\bx,\ldots,x_n-\bx)}<\varepsilon$, and
conditions \eqref{T4.3-5} are satisfied.
\end{enumerate}
If conditions \eqref{T4.3-5} in assertion \eqref{R5-1}$^\prime$ are replaced by conditions \eqref{C4.3-1}, the resulting assertion can be employed instead of assertion \eqref{C4.3-2} in Corollary~\ref{C4.3}.
In both cases, the second assertions of Theorem~\ref{T4.3} and Corollary~\ref{C4.3} remain valid.
\item
\label{R05.2}
Thanks to Proposition~\ref{P02.3}\,\eqref{P02.3.3}, instead of assuming in Theorem~\ref{T4.3} and Corollary~\ref{C4.3} the dual norm compatibility condition \eqref{C6}, one can assume that the primal norms $\|\cdot\|$ and $\vertiii{\,\cdot\,}$ are compatible in the sense of \eqref{C4}.
Moreover, in this case, the norm compatibility assumptions in the second parts of these statements can be dropped.
This observation is applicable also to the subsequent statements.
\item
Imposing certain sequential normal compactness assumptions (which are automatically satisfied in finite dimensions), one can formulate a limiting version of Theorem~\ref{T4.3} in terms of limiting normal cones; see comments after Theorem~\ref{EP}.
This remark applies also to the statements in the rest of the paper.
\end{enumerate}
\end{remark}

The Asplund space assertion \eqref{T4.2-2} in Theorem~\ref{T4.3} can be partially reversed.

\begin{theorem}
\label{T4.4}
Let $(X,\|\cdot\|)$ and $(X^n,\vertiii{\,\cdot\,})$ be Asplund,  $\widehat\Omega$ be closed, $\bx\in\bigcap_{i=1}^n\Omega_i$, $\al>0$ and $\be>0$.
Consider the following assertions:
\begin{enumerate}
\item\label{T4.4-1}
$\{\Omega_1,\ldots,\Omega_n\}$
is approximately $\al$-stationary at $\bx$;
\item\label{T4.4-2}
for any $\varepsilon>0$ and $\tau\in(0,1)$, there exist $\hat x:=(x_1,\ldots,x_n), (x'_1,\ldots,x'_n)\in\widehat\Omega$, $\hat a:=(a_1,\ldots,a_n)\in X^n$, {$x_0\in X$} and $\hat x^*:=(x_1^*,\ldots,x_n^*)\in N^F_{\widehat\Omega}(\hat x)$ such that conditions \eqref{T4.3-3}, \eqref{T4.3-5} and \eqref{T4.3-2} are satisfied;
\item\label{T4.4-3}
for any $\varepsilon>0$, there exist $\hat x:=(x_1,\ldots,x_n)\in\widehat\Omega$ and $\hat x^*:=(x_1^*,\ldots,x_n^*)\in N^F_{\widehat\Omega}(\hat x)$ such that
$\vertiii{(x_1-\bx,\ldots,x_n-\bx)}<\varepsilon$ and
conditions \eqref{T4.3-5} are satisfied.
\end{enumerate}
The following relations hold true:
{\renewcommand{\theenumi}{\rm\alph{enumi}}
\begin{enumerate}
\item\label{a1}
{\rm \eqref{T4.4-2} \folgt \eqref{T4.4-3}};
\item\label{b1}
if $\|\cdot\|$ and $\vertiii{\,\cdot\,}$  are compatible in the sense of \eqref{C4},
and $\be>\al$, then
{\rm \eqref{T4.4-1}~\folgt \eqref{T4.4-2}};
\item\label{c1}
if $\|\cdot\|$  and $\vertiii{\,\cdot\,}$  are compatible in the sense of \eqref{C5}, and $\al>\be$, then
{\rm \eqref{T4.4-3}~$\Rightarrow$~\eqref{T4.4-1}}.
\end{enumerate}
}
\end{theorem}

\begin{proof}
The implication in
\eqref{a1} is straightforward as \eqref{T4.4-3} is a simplified version of \eqref{T4.4-2}; see Remark~\ref{R05}\,\eqref{R05.1}.
In view of Proposition~\ref{P02.3}\,\eqref{P02.3.3},
the implication in \eqref{b1} is a direct consequence of Theorem~\ref{T4.3}\,\eqref{T4.2-2}.
We now prove the implication in \eqref{c1}.

Suppose that $\|\cdot\|$ and $\vertiii{\,\cdot\,}$  are compatible in the sense of \eqref{C5} with some $\kappa>0$,
assertion \eqref{T4.4-3} is true, and $\al>\be$.
Let $\eps>0$, $\hat x:=(x_1,\ldots,x_n)\in\widehat\Omega$ and $\hat x^*:=(x_1^*,\ldots,x_n^*)\in N^F_{\widehat\Omega}(\hat x)$ satisfy
$\vertiii{(x_1-\bx,\ldots,x_n-\bx)}<\varepsilon$ and
conditions \eqref{T4.3-5}.
Choose a $\xi\in(0,\al-\be)$.
Thus, $\xi':=\al-\be-\xi>0$.
By the definition of \Fr\ normal cone,
there is a $\rho\in(0,\varepsilon)$ such that
\begin{gather}
\label{T4.4P1}
\langle \hat x^*,\omega-\hat x\rangle\le \dfrac{\xi}{\kappa+\al} \vertiii{\omega-\hat x}<\xi\rho
\;\;
\text{for all}
\;\;
\omega\in\widehat\Omega\;\;\text{with}\;\;
\vertiii{\omega-\hat x}<(\kappa+\al)\rho.
\end{gather}
By the equality in \eqref{T4.3-5},
one can choose  an $\hat a:=(a_1,\ldots,a_n)\in X^n$ such that
\begin{gather}
\label{T4.4P2}
\vertiii{\hat a}<\al\rho
\;\;
\text{and}
\;\;
\langle \hat x^*,\hat a\rangle
>(\al-\xi')\rho=(\be+\xi)\rho.
\end{gather}
Suppose that $\bigcap_{i=1}^n(\Omega_i-x_i-a_i)\cap (\rho\B)\ne\emptyset$, and choose an $x_0\in\rho\B$ and $\hat x^\circ:=(x^\circ_1,\ldots,x^\circ_n)\in\widehat\Omega$ such that
$x^\circ_i-x_i-a_i=x_0$ for all $i=1,\ldots,n$.
Then, in view of \eqref{C5} and the first inequality \eqref{T4.4P2},
\begin{align*}
\vertiii{\hat x^\circ-\hat x}
\le \vertiii{(x_0,\ldots,x_0)_n}+\vertiii{\hat a}\le \kappa\|x_0\|+\vertiii{\hat a}<(\kappa+\al)\rho,
\end{align*}
and, by \eqref{T4.4P1}, $\ang{\hat x^*,\hat x^\circ-\hat x} <\xi\rho$.
Combining this with the second inequality in \eqref{T4.4P2}, we obtain
\begin{align*}
\Big\langle \sum_{i=1}^{n}x^*_i,x_0\Big\rangle=\langle \hat x^*,(x_0,\ldots,x_0)_n \rangle=\langle \hat x^*,\hat x^\circ-\hat x\rangle-\langle \hat x^*,\hat a\rangle<-\be\rho.
\end{align*}	
On the other hand, by the inequality in \eqref{T4.3-5},
$\langle \sum_{i=1}^{n}x^*_i,x_0\rangle> -\be\rho,$
a contradiction.
Hence, $\bigcap_{i=1}^n(\Omega_i-x_i-a_i)\cap (\rho\B)=\emptyset$, and consequently, $\{\Omega_1,\ldots,\Omega_n\}$
is approximate $\al$-stationary at $\bx$ with respect to $\vertiii{\,\cdot\,}$.
\end{proof}

\if{
\begin{remark}
Thanks to Proposition~\ref{P02.3}\,\eqref{P02.3.3}, all norm compatibility assumptions in Theorem~\ref{T4.4} can be replaced by the single assumption \eqref{C4}; cf. Remark~\ref{R05}\,\eqref{R05.2},
\end{remark}
}\fi

The next corollary gives the \emph{extended extremal principle}
generalizing and improving the corresponding result in \cite[Theorem~3.7]{Kru03}.

\begin{corollary}
\label{C4.5}
Let $(X,\|\cdot\|)$ and $(X^n,\vertiii{\,\cdot\,})$ be Asplund, $\|\cdot\|$ and $\vertiii{\,\cdot\,}$  be compatible in the sense of \eqref{C4}, $\widehat\Omega$ be closed, and $\bx\in\bigcap_{i=1}^n\Omega_i$.
The following assertions are equivalent:
\begin{enumerate}
\item
$\{\Omega_1,\ldots,\Omega_n\}$
is approximately stationary at $\bx$;
\item	
for any $\varepsilon>0$ and $\tau\in(0,1)$, there exist $(x_1,\ldots,x_n), (x'_1,\ldots,x'_n)\in\widehat\Omega$, $\hat a:=(a_1,\ldots,a_n)\in X^n$, {$x_0\in X$} and $\hat x^*:=(x_1^*,\ldots,x_n^*)\in N^F_{\widehat\Omega}(\hat x)$ such that conditions \eqref{T4.3-3}, \eqref{T4.3-2} and \eqref{C4.3-1} are satisfied;
\item
for any $\varepsilon>0$, there exist $\hat x:=(x_1,\ldots,x_n)\in\widehat\Omega$ and $\hat x^*:=(x_1^*,\ldots,x_n^*)\in N^F_{\widehat\Omega}(\hat x)$ such that
$\vertiii{(x_1-\bx,\ldots,x_n-\bx)}<\varepsilon$ and
conditions \eqref{C4.3-1} are satisfied.
\end{enumerate}
\end{corollary}

Reversing the conditions in Definition~\ref{D3.5}, we arrive at extensions of the \emph{transversality} properties discussed in \cite{Kru05,KruTha13,CuoKru21.2}.

\begin{definition}
Let $(X,\|\cdot\|)$ and $(X^n,\vertiii{\,\cdot\,})$ be normed spaces, and $\bx\in\bigcap_{i=1}^n\Omega_i$.
\begin{enumerate}
\item
Let $\al>0$.
The collection $\{\Omega_1,\ldots,\Omega_n\}$ is $\al$-transversal at $\bx$ if there exists an $\varepsilon>0$ such that
$\bigcap_{i=1}^n(\Omega_i-x_i-a_i)\cap (\rho\B)\ne\emptyset$ for all
$\rho\in(0,\varepsilon)$, $x_i\in\Omega_i$ and $a_i\in X$ $(i=1,\ldots,n)$ satisfying $\vertiii{(x_1-\bx,\ldots,x_n-\bx)}<\varepsilon$ and
$\vertiii{(a_1,\ldots,a_n)}<\al\rho$.
\item
The collection $\{\Omega_1,\ldots,\Omega_n\}$ is transversal at $\bx$ if it is $\al$-transversal at $\bx$ for some $\al>0$.
\end{enumerate}
\end{definition}

The statements of Theorem~\ref{T4.3} and its corollaries can also be easily ``reversed'' to produce dual characterizations of transversality.
For instance, Corollary~\ref{C4.5} leads to the following statement.

\begin{corollary}
Let $(X,\|\cdot\|)$ and $(X^n,\vertiii{\,\cdot\,})$ be Asplund, $\|\cdot\|$ and $\vertiii{\,\cdot\,}$  be compatible in the sense of \eqref{C4},  $\widehat\Omega$ be closed, and $\bx\in\bigcap_{i=1}^n\Omega_i$.
The collection $\{\Omega_1,\ldots,\Omega_n\}$ is  transversal at $\bx$ if and only if there exist $\al>0$ and $\varepsilon>0$ such that
$
\|\sum_{i=1}^{n}x^*_i\|>\al
$
for all $\hat x:=(x_1,\ldots,x_n)\in\widehat\Omega$ and $\hat x^*:=(x_1^*,\ldots,x_n^*)\in N^F_{\widehat\Omega}(\hat x)$
satisfying
$\vertiii{(x_1-\bx,\ldots,x_n-\bx)}<\varepsilon$ and
$\vertiii{\hat x^*}=1$.
\end{corollary}


Similar to approximate stationarity in Definition~\ref{D3.5}, one can formulate abstract product norm extensions of the extremality property in Definition~\ref{D1.1} and stationarity property studied in
\cite{Kru04,Kru05}.

\begin{definition}
\label{D4.6}
Let $(X,\|\cdot\|)$ and $(X^n,\vertiii{\,\cdot\,})$ be normed spaces, and $\bx\in\bigcap_{i=1}^n\Omega_i$.
The collection $\{\Omega_1,\ldots,\Omega_n\}$ is:
\begin{enumerate}
\item\label{D4.8-1}
extremal at $\bx$ if
there exists a $\rho\in(0,+\infty]$ such that,
for any $\varepsilon>0$, there exist  $a_1,\ldots,a_n\in X$ such that
$\vertiii{(a_1,\ldots,a_n)}<\varepsilon$ and
$\bigcap_{i=1}^n(\Omega_i-a_i)\cap B_\rho(\bx)=\emptyset$;
\item\label{D4.8-2}
stationary at $\bx$ if for any $\varepsilon>0$, there exist $\rho\in(0,\varepsilon)$ and $a_1,\ldots,a_n\in X$ such that
$\vertiii{(a_1,\ldots,a_n)}<\eps\rho$ and
$\bigcap_{i=1}^n(\Omega_i-a_i)\cap B_\rho(\bx)=\emptyset$.
\end{enumerate}
\end{definition}

It obviously holds \eqref{D4.8-1} $\Rightarrow$ \eqref{D4.8-2} in Definition~\ref{D4.6}, while the latter property implies approximate stationarity in Definition~\ref{D3.5}\,\eqref{D4.1-2}.
Thus, the necessary conditions for approximate stationarity in Corollary~\ref{C4.3} are applicable to the stationarity and extremality properties in Definition~\ref{D4.6}.
Observe also that global extremality (i.e., with $\rho=+\infty$) implies local extremality with any $\rho\in(0,+\infty)$.
By analogy with Definition~\ref{D3.5}\,\eqref{D4.1-1}, one can define quantitative versions of the properties in Definition~\ref{D4.6}.

The next proposition extending the corresponding assertions in \cite[Section~4.1]{Kru05} shows that in the convex case the main properties discussed above are equivalent.

\begin{proposition}
Let $(X,\|\cdot\|)$ and $(X^n,\vertiii{\,\cdot\,})$ be normed spaces, $\Omega_1,\ldots,\Omega_n$ be convex, and $\bx\in\bigcap_{i=1}^n\Omega_i$.
The following assertions are equivalent:
\begin{enumerate}
\item\label{P4.9-1}
$\{\Omega_1,\ldots,\Omega_n\}$ is locally extremal at $\bx$;
\item\label{P4.9-2}
$\{\Omega_1,\ldots,\Omega_n\}$ is stationary at $\bx$;
\item\label{P4.9-3}
$\{\Omega_1,\ldots,\Omega_n\}$ is approximately sta\-tionary at $\bx$;
\item\label{P4.9-4}
for any $\rho,\eps\in(0,+\infty)$, there exist  $a_1,\ldots,a_n\in X$ such that
$\vertiii{(a_1,\ldots,a_n)}<\varepsilon$ and
${\bigcap_{i=1}^n(\Omega_i-a_i)\cap B_\rho(\bx)=\emptyset}$.
\end{enumerate}	
\end{proposition}

\begin{proof}
Implications \eqref{P4.9-4} $\Rightarrow$ \eqref{P4.9-1} $\Rightarrow$ \eqref{P4.9-2} $\Rightarrow$ \eqref{P4.9-3} are direct consequences of the definitions.
We prove that
\eqref{P4.9-3} $\Rightarrow$ \eqref{P4.9-4}.

Suppose that \eqref{P4.9-4} does not hold, i.e., there exist $\varepsilon>0$ and $\rho>0$ such that $\bigcap_{i=1}^n(\Omega_i-a_i)\cap B_\rho(\bx)\ne\emptyset$ for all
$a_1,\ldots,a_n\in X$  satisfying
$\vertiii{(a_1,\ldots,a_n)}<\varepsilon$.
Then $\bigcap_{i=1}^n(\Omega_i-a_i-x_i)\cap(\rho\B) \ne\emptyset$ for all $x_i\in\Omega_i$ and
$a_i\in X$ $(i=1,\ldots,n)$ satisfying $\vertiii{(x_1-\bx,\ldots,x_n-\bx)}<\varepsilon/2$ and
$\vertiii{(a_1,\ldots,a_n)}<\varepsilon/2$.

Set $\eps':=\min\left\{\frac{\eps}{2},\frac{\varepsilon}{2\rho},\rho\right\}$.
Take arbitrary $\rho'\in(0,\eps')$,
$x_i\in\Omega_i$ and $a_i\in X$ $(i=1,\ldots,n)$  satisfying
$\vertiii{(x_1-\bx,\ldots,x_n-\bx)} <\varepsilon'$ and
$\vertiii{(a_1,\ldots,a_n)}<\rho'\eps'$.
Set $t:={\rho'}/{\rho}\in(0,1)$.
Then $\vertiii{(x_1-\bx,\ldots,x_n-\bx)} <\varepsilon/2$ and $\vertiii{(a_1/t,\ldots,a_n/t)}<\eps/2$.
Hence, there exists an $x\in\rho\B$ such that
$x+a_i/t+x_i\in\Omega_i$ for all $i=1,\ldots,n$.
Observe that $\|tx\|<t\rho=\rho'$ and, in view of convexity, $tx+a_i+x_i=t(x+a_i/t+x_i)+(1-t)x_i\in\Omega_i$ for all $i=1,\ldots,n$.
Thus, $tx\in\bigcap_{i=1}^n(\Omega_i-a_i-x_i)\cap (\rho'\B)$, i.e.,
$\{\Omega_1,\ldots,\Omega_n\}$ is not approximately stationary at $\bx$.
\end{proof}

\section*{Acknowledgments}

{We wish to thank the reviewers for their helpful comments which allowed us to improve the presentation.}

A part of the work was done during Alexander Kruger's stay at the Vietnam Institute for Advanced Study in Mathematics in Hanoi.
He is grateful to the Institute for its hospitality and supportive environment.

\section*{Disclosure statement}
No potential conflict of interest was reported by the authors.

\section*{Funding}
Nguyen Duy Cuong has been supported by the Postdoctoral Scholarship Programme of Vingroup Innovation Foundation (VinIF) code VINIF.2023.STS.54.

\section*{ORCID}
Nguyen Duy Cuong http://orcid.org/0000-0003-2579-3601
\\
Alexander Y. Kruger http://orcid.org/0000-0002-7861-7380

\addcontentsline{toc}{section}{References}
\bibliography{BUCH-kr,Kruger,KR-tmp}
\bibliographystyle{tfnlm}
\end{document}

%% file: macro.tex
\newcommand{\al}{\alpha}
\newcommand{\be}{\beta}
\newcommand{\ga}{\gamma}

\newcommand{\de}{\delta}
\newcommand{\eps}{\varepsilon}
\newcommand{\bx}{\bar x}

\newcommand{\iv}{^{-1} }

\newcommand {\R} {\mathbb R}
\newcommand {\N} {\mathbb N}

\newcommand {\B} {\mathbb B}

\newcommand {\dom} {{\rm dom}\,}
\newcommand {\epi} {{\rm epi}\,}

\newcommand {\sd} {\partial}

\newcommand{\folgt}{$ \Rightarrow\ $}

\newcommand{\vertiii}[1]{\left\vert\kern-0.25ex\left\vert\kern-0.25ex\left\vert #1\right\vert\kern-0.25ex\right\vert\kern-0.25ex\right\vert}
\newcommand{\vertiiiBig}[1]{\Big\vert\kern-0.25ex\Big\vert\kern-0.25ex\Big\vert #1\Big\vert\kern-0.25ex\Big\vert\kern-0.25ex\Big\vert}

\def\nbh{neighbourhood}
\def\es{\emptyset}

\def\RHS{right-hand side}

\def\EVP{Ekeland variational principle}
\def\Fr{Fr\'echet}

\newcommand{\norm}[1]{\left\Vert#1\right\Vert}
\newcommand{\abs}[1]{\left\vert#1\right\vert}

\newcommand{\ang}[1]{\left\langle #1 \right\rangle}

\newcommand{\AND}{\quad\mbox{and}\quad}
\newcounter{mycount}

\newcommand{\AK}[1]{\todo[inline]{AK {#1}}}

\newcommand\xqed{%
  \leavevmode\unskip\penalty9999 \hbox{}\nobreak\hfill
  \quad\hbox{$\triangleleft$}
  \smallskip}